\documentclass[a4paper,11pt,times,authoryears,reqno]{amsart}

\usepackage{geometry}
\usepackage{bm}
\usepackage{amssymb,amsfonts}
\usepackage{graphicx}
\usepackage{mathrsfs}
\usepackage{subcaption}
\usepackage{epstopdf}
\usepackage{float}
\usepackage{natbib}
\usepackage{enumerate}
\usepackage{booktabs}
\usepackage{adjustbox}
\usepackage{pdflscape}
\usepackage[colorlinks,citecolor=blue]{hyperref}

\newtheorem{theorem}{Theorem}[section]
\newtheorem{lemma}[theorem]{Lemma}

\theoremstyle{definition}

\theoremstyle{remark}

\newtheorem{proposition}[theorem]{Proposition}

\numberwithin{equation}{section}
\numberwithin{equation}{section}

\numberwithin{equation}{section}

\def\no{\nonumber}

  \def\no{\nonumber}

\def\bW{{\mathbf W}}

  \def\bD{{\mathbf D}} 
  
 \def\bH{{\mathbf H}}

 \def\bW{{\mathbf W}} \def\bX{{\mathbf X}}

  \def\bh{{\mathbf h}}

\def\bx{{\mathbf x}} \def\by{{\mathbf y}} 

\def\bbeta{{\boldsymbol{\beta}}}

 \def\bdelta{{\boldsymbol{\delta}}}

 \def\bGamma{{\boldsymbol{\Gamma}}}

 \def\bPhi{{\boldsymbol{\Phi}}}

\def\bSigma{{\boldsymbol{\Sigma}}}

\def\hbbeta{\widehat{\boldsymbol \beta}}

    \def\Cov{\hbox{Cov}}

\def\real{\mathop{{\rm I}\kern-.2em\hbox{\rm R}}\nolimits}

\def\1overn{\frac{1}{n}}

\def\bel{\begin{eqnarray}\label}  \def\eel{\end{eqnarray}}
\def\bes{\begin{eqnarray*}}  \def\ees{\end{eqnarray*}}

\def\PT{\hbbeta_{1,\tau}^{\rm PT}}
\def\FM{\hbbeta_{1,\tau}^{\rm FM}}
\def\SM{\hbbeta_{1,\tau}^{\rm SM}}
\def\PS{\hbbeta_{1,\tau}^{\rm PS}}
\def\SS{\hbbeta_{1,\tau}^{\rm S}}
\def\btau{\bbeta_{1,\tau}}
\def\bvt{{\boldsymbol \vartheta}}

\begin{document}

\title{Pretest and Stein-Type Estimations in Quantile Regression Model}


\author{Bahad{\i}r Y\"{u}zba\c{s}{\i}$^\dag$, Yasin Asar$^\S$, M.\c{S}amil \c{S}{\i}k$^\ddag$ and Ahmet Demiralp$^\clubsuit$}

\date{\today}
\maketitle

{\footnotesize
\center { \text{  $^\dag\ddag\clubsuit$ Department of Econometrics}\par
  { \text{ Inonu University}}\par
  {\text{ Malatya 44280, Turkey}}\par
  { \texttt{E-mail address: $^\dag$b.yzb@hotmail.com, $^\ddag$mhmd.sml85@gmail.com and $^\clubsuit$ahmt.dmrlp@gmail.com}}\par

  \vskip 0.2 cm

  \text{  $^\S$ Department of Mathematics-computer Sciences
}\par
  { \text{ Necmettin Erbakan University}}\par
  {\text{ Konya 42090, Turkey}}\par
  { \texttt{E-mail address: yasar@konya.edu.tr, yasinasar@hotmail.com}}

}}


\renewcommand\leftmark {\centerline{  \rm Quantile Shrinkage Estimation}}
\renewcommand\rightmark {\centerline{ \rm  Quantile Shrinkage Estimation}}

\renewcommand{\thefootnote}{}

\footnote {\par

Corresponding author : Bahad{\i}r Y\"{u}zba\c{s}{\i} 
}

\begin{abstract}
In this study, we consider preliminary test and shrinkage estimation strategies for quantile regression models. In classical Least Squares Estimation (LSE) method, the relationship between the explanatory and explained variables in the coordinate plane is estimated with a mean regression line. In order to use LSE, there are three main assumptions on the error terms showing white noise process of the regression model, also known as Gauss-Markov Assumptions, must be met: (1) The error terms have zero mean, (2) The variance of the error terms is constant and (3) The covariance between the errors is zero i.e., there is no autocorrelation. However, data in many areas, including econometrics, survival analysis and ecology, etc. does not provide these assumptions. Firstly introduced by Koenker, quantile regression has been used to complement this deficiency of classical regression analysis and to improve the least square estimation. The aim of this study is to improve the performance of quantile regression estimators by using pretest and shrinkage strategies. A Monte Carlo simulation study including a comparison with quantile $L_1$--type estimators such as Lasso, Ridge and Elastic Net are designed to evaluate the performances of the estimators. Two real data examples are given for illustrative purposes. Finally, we obtain the asymptotic results of suggested estimators.
\end{abstract}

\maketitle

\section{Introduction}
Consider a linear regression model
\begin{equation}
y_i=\bx_i'\bbeta+\varepsilon _{i},\ \ \
i=1,2,\dots,n,  \label{lin.mod}
\end{equation}%
where $y_{i}$'s are random responses, $\bx_i=\left(
x_{i1},x_{i2},\dots,x_{ip}\right)'$ are known vectors, $%
\bbeta =\left( \beta _{1},\beta _{2},\dots,\beta _{p}\right)
'$ is a vector of unknown regression coefficients, $\varepsilon _{i}{}$'$s$ are unobservable random errors and the superscript $\left( '\right) $ denotes the transpose of a vector or matrix. 

In a standard linear regression model, it is usually assumed that the residuals support Gauss--Markov assumptions. Then, one might want to investigate the average relationship between a set of regressors and the outcome variable based on the conditional mean function $E(y|x)$. Hence, the best linear unbiased estimator (BLUE) of the coefficients $\bbeta$ is given by the LSE, $\hbbeta^{\rm LSE}=(\bX'\bX)^{-1}\bX'\by$, where $\bX=(\bx_1,\dots,\bx_n)'$ and $\by=(y_1,\dots,y_n)'$. Not only, in real life, the assumptions of Gauss--Markov are not provided, but also the estimation of LSE provides only a partial view of the relationship, as we might be interested in describing the relationship at different points in the conditional distribution of $\by$. Furthermore, it doesn't give information about the relationship at any other possible points of interest and it is sensitive to outliers. In order to deal with these problems, there is a large literature. One of the most popular one is Quantile regression \cite{qr1978} has been an increasingly important modeling tool, due to its flexibility in exploring how covariates affect the distribution of the response. Quantile regression gives a more complete and robust representation of the variables in the model and doesn't make any distributive assumption about the error term in the model. The main advantage of quantile regression against LSE is its flexibility in modelling the data in heterogeneous conditional distributions.

On the basis of the quantile regression lies the expansion of the regression model to the conditional quantities of the response variable. A quantile is a one of the equally segmented subsets of a sorted sample of a population. It also corresponds to decomposition of distribution function with equal probability \cite{Gilchrist2000}. 
For a random variable $Y$  with distribution function $\mathcal{F}_{Y}\left(\by\right)=P( Y\le\by ) =\tau$ and
$0\le\tau\le 1$, the $\tau^{th}$ quantile function of $Y$, $\mathcal{Q}_{\tau}(\by)$, is defined to be
\begin{equation}\label{eq:y_tau}
\mathcal{Q}_{\tau}(Y\vert \bX)=\by_{\tau}=\mathcal{F}_{Y}^{-1}(\tau)=\inf\left\{\by\vert \mathcal{F}_{Y}\left(\by\right)\ge\tau\right\}\equiv \bx_i'\bbeta_\tau 
\end{equation}
where $\by_{\tau}$ is the inverse function of $\mathcal{F}_{Y}\left(\tau\right)$ for $\tau$ probability. In other words, the $\tau^{th}$ quantile in a sample corresponds to the probability $\tau$ for a $\by$ value.

As one can determine an estimation of conditional mean $E[\by\left\vert\bX=\bx\right]=\bx_i'\bbeta$ for a random sample $\by=\left(y_{1},y_{2},\dots,y_{n}\right)$ with empirical distribution function $\widehat{\mathcal{F}}_{Y}\left(\tau\right)$ by solving minimization of error squares also an estimation of the $\tau^{th}$  quantile regression coefficients ($\hbbeta_{\tau}$) can be defined by solving the following minimization of absolute errors problem 
\begin{equation}
\underset{\bbeta\in\Re^{p}}{\arg \min} \sum_{i=1}^{n}\rho_\tau (y_i-\bx'_{i}\bbeta),
\end{equation}
where $\rho_\tau(u)=u(\tau-I(u<0))$ is the quantile loss function. Hence, it yields
\begin{equation}\label{eq:est_tau}
\hbbeta_{\tau}=\underset{\bbeta\in\Re^{p}}{\arg \min}\Bigg[ \sum_{i\in\left\{i: y_i\ge \bx_i^{'}\bbeta\right\}}^{n}
\tau\vert\ y_i-\bx'_{i}\bbeta\vert - \sum_{i\in\left\{i: y_i\ <\bx'_i\bbeta\right\}}^{n}
(1-\tau)\vert\ y_i-\bx'_{i}\bbeta\vert\Bigg].
\end{equation}

The quantile regression concept was firstly introduced by \cite{qr1978}. In their study, they identified the autoregression quantile term which is a generalization to a linear model of a basic minimization problem which generates sample quantiles and they produce some homoscedasticity properties and joint asymptotic distribution of the estimator. Thus they made it possible to linear model generalizations of some known strong position estimators for sorted data. After its first presentation, the  quantile regression has become an important and commonly used method and turned into an important tool of applied statistics in the last three decades.

\cite{sen:saleh1987} proposed pretest and Stein-type estimations based on M-type estimator in multiple linear regression models. \cite{Koenker2008} proposed the \textit{quantreg} R package and it is implementations for linear, non-linear and non-parametric quantile regression models.  R and the package \textit{quantreg} are open-source software projects and can be freely downloaded from CRAN: \url{http://cran.r-project.org}. \cite{wei2012} proposed a shrinkage estimator that automatically regulates the possible deviations to protect the model against making incorrect parameter estimates, when some independent variables are randomly missing in a quantile regression model and they tested the performance of this estimator on a finite sample using real data. The books by \cite{Koenker2005} and \cite{Davino2014} are excellent sources for various properties of Quantile Regression as well as many computer algorithms. \cite{hqreg} developed an efficient and scalable algorithm for computing the solution paths for these models with the Zou and Hastie's elastic-net penalty. They also provide an implementation via the R package \textit{hqreg} publicly available on CRAN. Furthermore, it provides Tibshirani's Lasso and Hoerl and Kennard's Ridge estimators based on Quantile Regression models.
The \textit{hqreg} function uses an approximate optimal model instead of calculating the exact value of a single result for a given value, unlike Koenker's \textit{quantreg} function. The package \textit{hqreg} is also an open-source software project and can be freely downloaded from \url{https://cloud.r-project.org/web/packages/hqreg/index.html}.

\cite{ahmed2014} provided a collection of topics outlining pretest and Stein-type shrinkage estimation techniques in a variety of regression modeling problems. Recently, \cite{NA2017} proposed preliminary and positive rule Stein-type estimation strategies based on ridge Huberized estimator in the presence of multicollinearity and outliers. \cite{yuzbasi-ahmed2016} considered ridge pretest, ridge shrinkage and ridge positive shrinkage estimators for a semi-parametric model when the matrix $\bX'\bX$ appears to be ill-conditioned and the coefficients $\bbeta$ can be partitioned as $\left(\bbeta_1,\bbeta_2\right)$, especially when $\bbeta_2$ is close to zero. \cite{yuzbasi-et-al2017} combined ridge estimators with pretest and shrinkage estimators for linear regression models. \cite{AutoReg} very recently suggested quantile shrinkage estimation strategies when the errors are autoregressive.

The aim objective of this study is to improve the performance of quantile regression estimators by combining the idea of pretest and shrinkage estimation strategies with quantile type estimators. Hence, the paper is organized as follows. The full and sub-model estimators based on quantile  regression are given in Section~\ref{ES}. Moreover, the pretest, shrinkage estimators and penalized estimations are also given in this section. The asymptotic properties of the pretest and shrinkage estimators estimators are obtained in Section~\ref{AA}. The design and the results of a Monte Carlo simulation study including a comparison with other penalty estimators are given in Section ~\ref{sim}. Two real data examples are given for illustrative purposes in Section ~\ref{RAA}.  The concluding remarks are presented in Section ~\ref{conc}.

\section{Estimation Strategies}
\label{ES}

In this study we consider the estimation problem for quantile regression models when there are many potential predictors and some of them may not have influence on the response of interest. The analysis will be more precise if ``irrelevant variables" can be left out of the model. As in linear regression problems, this leads us to two different sets of model: The full model where $\bx'_{i}\bbeta$ includes all the $\bx_i$'s with available data; and a candidate sub-model set that includes the predictors of main concern while leaving out irrelevant variables. That is, we consider a candidate subspace where an unknown $p$-dimensional parameter vector $\bbeta$ satisfies a set of linear restrictions 
\begin{equation*}
\bH\bbeta=\bh
\end{equation*}%
where $\bH$ is a $p_2\times p$ matrix of rank $p_2\le p$ and $\bh$ is a given $p_2\times 1$ vector of constants.
Let us denote the full model estimators by $\widehat\bbeta_{\tau}^{\rm FM}$ and the sub-model estimators by $\widehat\bbeta_{\tau}^{\rm SM}$.  
One way to deal with this framework is to use pretest procedures that test whether the coefficients of the irrelevant variables are zero and then estimate parameters in the model that include coefficients that are rejected by the test. Another approach is to use Stein type shrinkage estimators where the estimated regression coefficient vector is shrunk in the direction of the candidate subspace.

\subsection{Full and Sub-Model Estimations}

Linear regression model in $\eqref{lin.mod}$ can be written in a partitioned form as follows
\begin{equation}
y_i=\bx_{1i}'\bbeta_1+\bx_{2i}'\bbeta_2+\varepsilon _{i},\ \ \
i=1,2,\dots,n,  \label{part.lin}
\end{equation}
where $p=p_1+p_2$ and $\bbeta_1$, $\bbeta_2$ parameters are of order $p_1$ and $p_2$, respectively. $\bx_i=\left(\bx_{1i}',\bx_{2i}'\right)$ and $\varepsilon _{i}$'s are errors with the same joint distribution function $\mathcal{F}$.

The conditional quantile function of response variable $y_i$  would be written as follows
\begin{equation}\label{eq:x_tau}
\mathcal{Q}_{\tau}(y_i\vert \bx_i)=\bx_{1i}'\bbeta_{1,\tau}+\bx_{2i}'\bbeta_{2,\tau},\ \ \ 0<\tau< 1 
\end{equation}
The main interest here is to test the null hypothesis
\begin{equation}
\label{NullHp}
H_0:\bbeta_{2,\tau}= \bold{0}_{p_2}.
\end{equation}%

We assume that the sequence of design matrices $\bX$ satisfies the following conditions
\begin{enumerate}
\item[(A1)] $\lim_{n\rightarrow\infty} \frac{1}{n}\sum_{i=1}^n\ \bx_i\bx_i'=\bD,\ \ \
\bD_{0}=\frac{1}{n}\bX'\bX$
\item[(A2)] $\max_{1\leq i\leq n}\Vert{\bx_i}\Vert\Big/\sqrt{n}\rightarrow\ 0$
\end{enumerate}
where $\bD_0$ is a positive definite matrix.%

In order to test \eqref{NullHp}, under the assumptions (A1-A2), we consider the following Wald test statistics is given by
\begin{equation}\label{wald}
\mathcal{W}=nw^{-2}\left(\hbbeta_{2,\tau}^{\rm FM}\right)'\left(\bD^{22}\right)^{-1}\hbbeta_{2,\tau}^{\rm FM}
\end{equation}
where $\bD_{ij}:i,j=1,2$ is the $\left(i,j\right)^{th}$ partition of the $\bD$ matrix and $\bD^{ij}$ is the $\left(i,j\right)^{th}$ partition of the $\bD^{-1}$ matrix. Also $w=\sqrt{\tau\left(1-\tau\right)}\Big/f\left(\mathcal{F}^{-1}\left(\tau\right)\right)$ and $\bD^{22}=\left(\bD_{22}-\bD_{21}\bD_{11}^{-1}\bD_{12}\right)^{-1}$. Here, the term $f\left(\mathcal{F}^{-1}\left(\tau\right)\right)$ , which plays the role of the distress parameter, is generally called the sparsity function \cite{Cooley-Tukey1965} or quantile density function \cite{Parzen1979}.The sensitivity of the test statistic naturally depends on this parameter. The distribution of $\mathcal{W}$  fits the chi-square distribution with $p_2$ degree of freedom under the null hypothesis.

Full model quantile regression estimator is the value that minimizes the following problem
\begin{equation*}
\hbbeta_{\tau}^{\rm FM}=\underset{\bbeta\in\Re^{p}}{\min} \sum_{i=1}^{n}\rho_\tau (y_i-\bx'_{i}\bbeta)
\end{equation*}%

Sub-model (SM) quantile regression estimator of $\bbeta_{\tau}$ is given by \begin{equation*}\hbbeta_{\tau}^{\rm SM}=\left(\hbbeta_{1,\tau}^{\rm SM},\bm0_{p_2}\right)\end{equation*}%
Also it is the value value that minimizes the following problem
\begin{equation*}
\hbbeta_{1,\tau}^{\rm SM}=\underset{\bbeta_1\in\Re^{p_1}}{\min} \sum_{i=1}^{n}\rho_\tau (y_i-\bx_{1i}'\bbeta_1)
\end{equation*}%

\subsection{Pretest and Stein-Type Estimations}

The pretest was firstly applied by \cite{Bancroft1944} for the validity of the unclear preliminary information $(\rm UPI)$ by subjecting it to a preliminary test and according to this verification adding it to the model as a parametric constraint in the process of selecting the information between the sub model and the full model. In the pretest method, after a sub model estimator obtained in addition to obtaining the full model estimator  the validity of the subspace information would be tested by using an appropriate  test statistic $\mathcal{W}$.

The preliminary test estimator $\left(\PT \right)$ could be obtained by following equation
\begin{equation}\label{eq:pre}
\PT =\hbbeta_{1,\tau}^{\rm FM}-\left(\hbbeta_{1,\tau}^{\rm FM}-\hbbeta_{1,\tau}^{\rm SM}\right)\textrm{I}\left(\mathcal{W}<\chi^2_{p_2,\alpha}\right)
\end{equation}
where $\textrm{I}\left(*\right)$ is an indicator function and $\chi^2_{p_2,\alpha}$ is the $100\left(1-\alpha\right)$ percentage point of the $\mathcal{W}$.

The Stein-type shrinkage (S) estimator is a combination of the over--fitted model estimator $\bm{\widehat{\beta}}_{1,\tau}^{\rm FM}$ with the under--fitted $\bm{\widehat{\beta}}_{1,\tau}^{\rm SM}$, given by
\begin{equation*}
\bm{\widehat{\beta}}_{1,\tau}^{\textrm{S}}=\bm{\widehat{\beta}}_{1,\tau}^{\rm FM}-d\left( \bm{\widehat{\beta}}_{1,\tau}^{\rm FM}-\bm{\widehat{\beta}}_{1,\tau}^{\rm SM}\right)\mathcal{W}_{n}^{-1} \text{, } d=(p_2-2)\geq 1,
\end{equation*}

In an effort to avoid the over-shrinking problem inherited by $\bm{\widehat{\beta}}_{1,\tau}^{\textrm{S}}$ we suggest using the positive part of the shrinkage estimator defined by
\begin{eqnarray*}
\hbbeta_{1,\tau}^{\textrm{PS}}&=&\hbbeta_{\tau}%
^{\rm SM}+\left( 1-d\mathcal{W}_{n}^{-1}\right)\textrm{I}\left(\mathcal{W}_{n} >  d\right)\left( \hbbeta_{1,\tau}^{\rm FM}-\bm{%
\widehat{\beta }}_{1,\tau}^{\rm SM}\right)\cr
&=&\bm{\widehat{\beta}}_{1,\tau}^{\textrm{S}}-\left( \hbbeta_{1,\tau}^{\rm FM}-\bm{%
\widehat{\beta }}_{1,\tau}^{\rm SM}\right)
\left( 1-d\mathcal{W}_{n}^{-1}\right)\textrm{I}\left(\mathcal{W}_{n}\leq  d\right).
\end{eqnarray*}%

\subsection{Penalized Estimation}
The penalized estimators for quantile are given by \cite{hqreg}.
\begin{equation}
\hbbeta^{\rm Penalized}_{\tau}=\underset{\bm{\bm{\beta}}}{\arg \min}\sum_{i}\rho(y_i-\bx_{i}'\bbeta)+\lambda\ P(\bbeta)
\end{equation}
where $\rho$ is a quantile loss function, $P$ is a penalty function and $\lambda$ is a tuning parameter.  
\begin{equation*}
P(\bbeta)\equiv P_{\alpha}(\bbeta)=\alpha\Vert\bbeta\Vert_1+\frac{(1-\alpha)}{2}\Vert\bbeta\Vert_2^2
\end{equation*}
which is the lasso penalty for $\alpha=1$ (\cite{lasso}), the ridge penalty for $\alpha=0$ (\cite{Ho-Ke1970}) and the elastic-net penalty for $0\le \alpha\le {1}$ (\cite{elastic-net}).

\section{Asymptotic Analysis}
\label{AA}

In this section, we demonstrate the asymptotic properties of suggested estimators. First, we consider the following theorem. 

\begin{theorem}\label{teo_dist_FM} The distribution of the quantile regression full model with i.i.d. variables and under assumptions A1 and A2,
\begin{equation}
\sqrt{n}(\FM-\bbeta_\tau)\rightarrow^{\hspace{-0.3cm}d}\mathcal{N}(0,w^2\bD^{-1})
\end{equation}
where $\rightarrow^{\hspace{-0.3cm}d}$ denotes convergence in distribution as $n\rightarrow\infty$.
\end{theorem}
\begin{proof}
The proof can be obtained from \cite{Koenker2005}.
\end{proof}

Let $\left\{K_n\right\}$ be a sequence of local alternatives  given by
\begin{equation*}
\ K_n:\boldsymbol{\beta}_{2,\tau}=\frac{\boldsymbol{\kappa}}{\sqrt{n}}
\end{equation*}
where $\boldsymbol{\kappa}=\left(
\kappa_{1},\kappa_{2},\dots,\kappa_{p_{2}}\right)'\in \Re^{p_{2}}$ is a fixed vector. If $\boldsymbol{\kappa}=\bold{0}_{p_2}$, then the null hypothesis is true. Furthermore, we consider the following proposition to establish the
asymptotic properties of the estimators.

\begin{proposition}\label{prop_vector_dist} 
Let $\bvt_{1} = \sqrt{n}\left(\hbbeta_{1, \tau}^{\rm FM}-%
\bbeta_{1,\tau}\right)$, $\bvt_{2} =\sqrt{n}\left( \hbbeta_{1,\tau}^{\rm SM}-%
\bbeta_{1,\tau}\right)$ and $\bvt_{3} =\sqrt{n}\left( \hbbeta_{1,\tau}^{\rm FM}-%
\hbbeta_{1,\tau}^{\rm SM}\right)$. Under the regularity assumptions A1 and A2, Theorem~\ref{teo_dist_FM} and the local alternatives $\left\{ K_{n}\right\}$, as $n\rightarrow \infty$ we have the following joint distributions:

$$\left(
\begin{array}{c}
\bvt_{1} \\
\bvt_{3}%
\end{array}%
\right) \sim\mathcal{N}\left[ \left(
\begin{array}{c}
\boldsymbol{0 }_{p_1} \\
-\boldsymbol{\delta }%
\end{array}%
\right) ,\left(
\begin{array}{cc}
w^2\bD_{11.2}^{-1} & \bSigma_{12} \\
\bSigma_{21} & \bPhi%
\end{array}%
\right) \right],$$


$$\left(
\begin{array}{c}
\bvt_{3} \\
\bvt_{2}%
\end{array}%
\right) \sim\mathcal{N}\left[ \left(
\begin{array}{c}
-\boldsymbol{\delta } \\
\boldsymbol{\delta }%
\end{array}%
\right) ,\left(
\begin{array}{cc}
\bPhi & \bSigma^* \\
\bSigma^* & w^2 \bD_{11}^{-1}%
\end{array}%
\right) \right],$$ \\
where $\bdelta=\bD_{11}^{-1}\bD_{12}\boldsymbol{\kappa}$,  $\bPhi=\omega^2\bD_{11}^{-1}\bD_{12}\bD_{22.1}^{-1}\bD_{21}\bD_{11}^{-1}$, $\bSigma_{12}=-\omega^2\bD_{12}\bD_{21}\bD_{11}^{-1}$ and $\bSigma^*=\bSigma_{21}+w^2\bD_{11.2}^{-1}$.
\end{proposition}
\begin{proof}
See Appendix.
\end{proof}

Now, we are ready to obtain the asymptotic distributional biases of estimators which are given the following section.

\subsection{The Performance of Bias} 
The asymptotic distributional bias of an estimator $\hbbeta_{1, \tau}^{\ast }$ is defined as
\begin{eqnarray*}
\mathcal{B}\left( \hbbeta_{1,\tau}^{\ast }\right) =\mathbb{E} \underset{%
n\rightarrow \infty }{\lim }\left\{\sqrt{n}\left( \hbbeta_{1,\tau}^{\ast }-%
\bbeta_{1,\tau} \right) \right\}.
\end{eqnarray*}

\begin{theorem}
Under the assumed regularity conditions A1 and A2, the Proposition \ref{prop_vector_dist},  the Theorem~\ref{teo_dist_FM}  and the local alternatives $\left\{K_n\right\}$, the expressions for asymptotic biases for listed estimators are:
\label{bias}
\begin{eqnarray*}
\mathcal{B}\left(\hbbeta_{1,\tau}^{\rm FM}\right) &=&
\boldsymbol{0}\\
\mathcal{B}\left( \hbbeta_{1,\tau}^{\rm SM}\right) &=&%
\boldsymbol{\delta }\\
\mathcal{B}\left( \hbbeta_{1,\tau}^{\rm PT}\right) &=&\bdelta H_{p_{2}+2}\left( \chi _{p_{2},\alpha
}^{2};\Delta \right)  \\
\mathcal{B}\left( \hbbeta_{1,\tau}^{\rm S}\right) &=&\bdelta \left\{ d\mathbb{E}\left\{\chi _{p_{2}+2}^{-2}(\Delta^2) \right\}+\mathbb{H}_{p_{2}+2}\left( d;\Delta \right)\right. \\
&&-\left. d\mathbb{E}\left( \chi _{p_{2}+2}^{-2}(\Delta^2)\textrm{I}\left(\chi _{p_{2}+2}^{2}(\Delta^2)< d\right) \right) \right\}
\end{eqnarray*}%
where $\Delta = \boldsymbol{\kappa}'\left(w^{2}\boldsymbol{D}_{22.1}^{-1}\right)^{-1}\boldsymbol{\kappa}$, $d=p_2-2$ and $\mathbb{H}_{v}\left( x,\Delta \right) $ is the cumulative distribution
function of the non-central chi-squared distribution with non-centrality
parameter $\Delta $ and $v$ degree of freedom, and
\begin{equation*}
\mathbb{E}\left( \chi _{v}^{-2j}\left( \Delta \right) \right)
=\int\nolimits_{0}^{\infty }x^{-2j}d\mathbb{H}_{v}\left( x,\Delta \right) .
\end{equation*}

\end{theorem}
\begin{proof}
See Appendix.
\end{proof}

Now, we define the following asymptotic quadratic bias $\left(\mathcal{QB}\right)$ of an estimator $\hbbeta_{1, \tau}^{\ast }$
by converting them into the quadratic form since the bias expression of all the estimators are not in the scalar form.
\begin{equation}\label{eq:asqubi}
\mathcal{QB}\left(\hbbeta_{1, \tau}^{\ast }\right)=\mathcal{B}\left(\hbbeta_{1, \tau}^{\ast }\right)'\bD_{11.2}\mathcal{B}\left(\hbbeta_{1, \tau}^{\ast }\right)
\end{equation}

Using the definition given in $\eqref{eq:asqubi}$, the asymptotic distributional quadratic bias of the
estimators are presented below.
\begin{eqnarray*}
\mathcal{QB}\left(\hbbeta_{1,\tau}^{\rm FM}\right) &=&
\boldsymbol{0}\\
\mathcal{QB}\left( \hbbeta_{1,\tau}^{\rm SM}\right) &=&
\boldsymbol{\delta}'\bD_{11.2}\boldsymbol{\delta} \\
\mathcal{QB}\left( \hbbeta_{1,\tau}^{\rm PT}\right) &=&\boldsymbol{\delta}'\bD_{11.2}\boldsymbol{\delta} \Big[H_{p_{2}+2}\left( \chi _{p_{2},\alpha
}^{2};\Delta \right)\Big]^2 \\
\mathcal{QB}\left( \hbbeta_{1,\tau}^{\rm S}\right) &=& d^2\boldsymbol{\delta}'\bD_{11.2}\boldsymbol{\delta}\Big[\mathbb{E}\left\{\chi _{p_{2}+2
}^{-2}\left( \Delta\right)\right\}\Big]^2 \\
\mathcal{QB}\left( \PS \right) &=&\bdelta'\bD_{11.2}\bdelta \Big[d\mathbb{E}\left\{\chi _{p_{2}+2}^{-2}(\Delta^2) \right\}+\mathbb{H}_{p_{2}+2}\left( d;\Delta \right) \Big.\\
&&-\Big.d\mathbb{E}\left( \chi _{p_{2}+2}^{-2}(\Delta^2)\textrm{I}\left(\chi _{p_{2}+2}^{2}(\Delta^2)< d\right) \right)\Big]^2
\end{eqnarray*}%

The $\mathcal{QB}$ of $\hbbeta_{1,\tau}^{\rm FM}$ is $\boldsymbol 0$ and the $\mathcal{QB}$ of $\hbbeta_{1,\tau}^{\rm SM}$ is an unbounded function of $\boldsymbol{\delta}'\bD_{11.2}\boldsymbol{\delta}$. The $\mathcal{QB}$ of $\hbbeta_{1,\tau}^{\rm PT}$ starts from $\boldsymbol 0$ at $\Delta=0$, and when $\Delta=0$ increases it to the maximum point and then decreases to zero. For the $\mathcal{QB}$s of $\hbbeta_{1,\tau}^{\rm S}$ and $\hbbeta_{1,\tau}^{\rm PS}$, they similarly start from $\boldsymbol 0$, and increase to a
point, and then decrease towards zero. 


\subsection{The performance of Risk}
The asymptotic distributional risk of an estimator $\hbbeta_{1, \tau}^{\ast }$ is defined as
\begin{equation}
\label{def:risk}
\mathcal{R}\left( \hbbeta_{1,\tau}^{\ast} \right) = {\rm tr}\left( \boldsymbol{W}\bGamma(\hbbeta_{1,\tau}^{\ast}) \right) 
\end{equation}
where $\boldsymbol{W}$ is a positive definite matrix of weights with dimensions of $p_1 \times p_1$, and $\boldsymbol{\Gamma}$ is the asymptotic covariance matrix of an estimator $\hbbeta_{1, \tau}^{\ast }$ defined as
\begin{equation*}
\boldsymbol{\Gamma} \left( \hbbeta_{1,\tau}^{\ast} \right) = \mathbb{E}\left\{\underset{n\rightarrow \infty }{\lim }{n}\left( \hbbeta_{1,\tau}^{\ast} -\bbeta_{1,\tau}\right)\left( \hbbeta_{1,\tau}^{\ast}-\bbeta_{1,\tau}\right) ^{'}\right\}.
\end{equation*}





\begin{theorem}\label{risks}
Under the assumed regularity conditions in A1 and A2, the Proposition \ref{prop_vector_dist},  the Theorem~\ref{teo_dist_FM}  and $\left\{K_n\right\}$, the expressions for asymptotic risks for listed estimators are:
\begin{eqnarray*}
\mathcal{R}\left( \FM\right)
&=&w^2{\rm tr}\left(\bW\bD_{11.2}^{-1}\right)\\
\mathcal{R}\left( \SM \right)
&=&w^2{\rm tr}\left(\bW \bD_{11}^{-1}\right)
+\bdelta' \bW \bdelta\\
\mathcal{R}\left( \PT\right) 
&=&\mathcal{R}\left(\FM \right)+{\rm tr}\left(\bW\bdelta \bdelta ^{'}\bPhi^{-1}\bSigma_{21} \right)\left[\mathbb{H}_{p_{2}+4}\left( \chi _{p_{2},\alpha}^{2};\Delta \right)\right.\\
&&+\left. \mathbb{H}_{p_{2}+2}\left( \chi _{p_{2},\alpha}^{2};\Delta \right)\right]+{\rm tr}\left(\bW\bPhi\right) \mathbb{H}_{p_{2}+2}\left(
\chi^2_{p_2,\alpha};\Delta \right)\\
&& +\bdelta'\bW \bdelta\mathbb{H}_{p_{2}+4}\left( \chi^2_{p_2,\alpha};\Delta \right)
-2{\rm tr}\left(\bW\bSigma_{21}\right)\mathbb{H}_{p_{2}+2}\left(
\chi _{p_{2},\alpha }^{2};\Delta \right)\\
\mathcal{R}\left( \SS \right) 
&=&\mathcal{R}\left(
\FM\right) -2d{\rm tr}\left(\bW \bSigma_{21}\right) \mathbb{E}\left\{ \chi _{p_{2}+2
}^{-2}\left( \Delta \right) \right\}\\
&&-2d {\rm tr}\left(\bW\bdelta\bdelta^{'}\bPhi\bSigma_{21}\right)\mathbb{E}\left\{\chi _{p_{2}+4
}^{-2}\left( \Delta\right)\right\} \\
&&+d^2{\rm tr}\left(\bW\bPhi\right)\mathbb{H}_{p_{2}+2}\left(
\chi _{p_{2},\alpha }^{2};\Delta \right)+d^2 \bdelta \bW\bdelta \mathbb{H}_{p_{2}+4}\left( \chi _{p_{2},\alpha
}^{2}\left(\Delta \right)\right) \\
\mathcal{R}\left( \PS \right)  
&=&\mathcal{R}\left(
\SS \right)-2{\rm tr}\left(\bW\bSigma_{21}\right)\mathbb{E}\left(1-d\chi _{p_{2}+2 }^{-2}\left( \Delta \right)\right)\textrm{I}\left( \chi _{p_{2}+2 }^{2}\left( \Delta \right)<d\right)  \\ 
&&+2{\rm tr}\left(\bW\bdelta\bdelta^{'}\bPhi^{-1}\bSigma_{21}\right)\mathbb{E}\left(1-d\chi _{p_{2}+4 }^{-2}\left( \Delta \right)\right)\textrm{I}\left(\chi_{p_{2}+4}^{2}\left( \Delta \right)<d\right) \\
&&+2{\rm tr}\left(\bW\bdelta\bdelta^{'}\bPhi^{-1}\bSigma_{21}\right)\mathbb{E}\left(1-d\chi _{p_{2}+2 }^{-2}\left( \Delta \right)\right)\textrm{I}\left(\chi_{p_{2}+2}^{2}\left( \Delta \right)<d\right) \\
&&-d^2{\rm tr}\left(\bW\bPhi\right) \mathbb{E}\left\{\chi _{p_{2}+2}^{-4}\left(\Delta\right)\right\}\textrm{I}\left(\chi _{p_{2}+2}^{2}\left(\Delta\right)\leq d\right)\no \\
&&-d^2\bdelta'\bW\bdelta\mathbb{E}\left\{\chi _{p_{2}+2}^{-4}\left(\Delta\right)\right\}\textrm{I}\left(\chi _{p_{2}+2}^{2}\left(\Delta\right)\leq d\right)\no \\
&&+{\rm tr}\left(\bW\bPhi\right) \mathbb{H}_{p_{2}+2}\left(
d;\Delta \right)+\bdelta'\bW\bdelta\mathbb{H}_{p_{2}+4}\left(
d;\Delta \right)
\end{eqnarray*}%
\end{theorem}

Noting that if $\bD_{12}= \boldsymbol{0}$, then all the risks reduce to common value  \linebreak $\omega^2{\rm tr}\left(\bW \bD_{11}\right)$ for all $\bW$. For  $\bD_{12}\neq 0$, the risk of  $\FM$ remains constant while the risk of $\SM$ is an bounded function of $\Delta$ since $\Delta \in [0,\infty]$. The risk of $\PT$ increases as $\Delta$ moves away from zero, achieves it maximum and then decreases towards the risk of the full model estimator. Thus, it is a bounded function of $\Delta$. The risk of $\hbbeta_{1,\tau}^{\rm FM}$ is smaller than the risk of $\hbbeta_{1,\tau}^{\rm PT}$ for some small values of $\Delta$ and opposite conclusions holds for rest of the parameter space. It can be seen that $\mathcal{R}\left(\PS\right)\leq{\mathcal{R}\left(
\hbbeta_{1,\tau}^{\rm S}\right)}\leq{\mathcal{R}\left(\FM\right)},$ 
strictly inequality holds for small values of $\Delta$. Thus, positive shrinkage is superior to the shrinkage estimator. However, both shrinkage estimators outperform the full model estimator in the entire parameter space induced by $\Delta$. On the other hand, the pretest estimator performs better than the shrinkage estimators when $\Delta$ takes small values and outside this interval the opposite conclusion holds.

\section{Simulations}
\label{sim}
We conduct Monte-Carlo simulation experiments to study the performances of the proposed estimators under various practical settings. In order to generate the response variables, we use
\begin{equation*}
y_i=\bx_i'\bbeta+\sigma\varepsilon_i,\ i=1,\dots,n,
\end{equation*}%
where $\bx_i$'s are standard normal. The correlation between the $j$th and $k$th components of $\bx$ equals to $0.5^{|j-k|}$ and also $\varepsilon_i$'s follow i.i.d.

\subsection{Asymptotic Investigations}
 We consider that the regression coefficients are set
$\bbeta=\left( \bbeta_{1}',\bbeta_{2}'\right)' =\left( \mathbf{1}'_{p_1},\mathbf{0}_{p_2}'\right)'$, where $\mathbf{1}_{p_1}$ and $\mathbf{0}_{p_2}$ mean the vectors of 1 and 0 with dimensions $p_1$ and $p_2$, respectively. In order to investigate the behavior of the estimators, we define $\Delta^{\ast}=\left\Vert \bbeta -\bbeta_{0}\right\Vert $, where $\bbeta_{0}=\left( \bold{1}'_{p_1},\mathbf{0}_{p_2}'\right)'$ and $\left\Vert \cdot \right\Vert $ is the Euclidean norm. If $\Delta^{\ast}=0$, then it means that we will have $\bbeta=\left( \mathbf{1}_{p_1}',\mathbf{0}_{p_2}'\right)'$ to generate the response while we will have $\bbeta=(\mathbf{1}'_{p_1},2, \mathbf{0}_{p_{2}-1}')'$ when $\Delta^{\ast}>0$, say $\Delta^{\ast}=2$. When we increase the number of $\Delta^{\ast}$, it indicates the degree of the violation of null hypothesis. Also, we consider that $n=60$, $p_1=5$, $p_2=5,10$ and $\alpha=0.01,0.05,0.10,0.25$. Furthermore, we consider both $\sigma=1,3$ and errors are only taken from standard normal distribution. In this case we investigate the performance of suggested estimators 
for different values of $\Delta^{\ast}$.

The performance of an estimator $\hbbeta_{\tau}^{\ast}$ was evaluated by
using the model error (ME) criterion which is defined by
\begin{equation*}
\textnormal{ME}\left(\hbbeta_{\tau}^{\ast}\right) =\left( \hbbeta_{\tau}^{\ast}-\bbeta\right)'\left( \hbbeta_{\tau}^{\ast}-\bbeta\right)
\end{equation*}

\begin{figure}
\centering
   \begin{subfigure}[a]{.9\textwidth}
   \includegraphics[height=8cm,width=14cm]{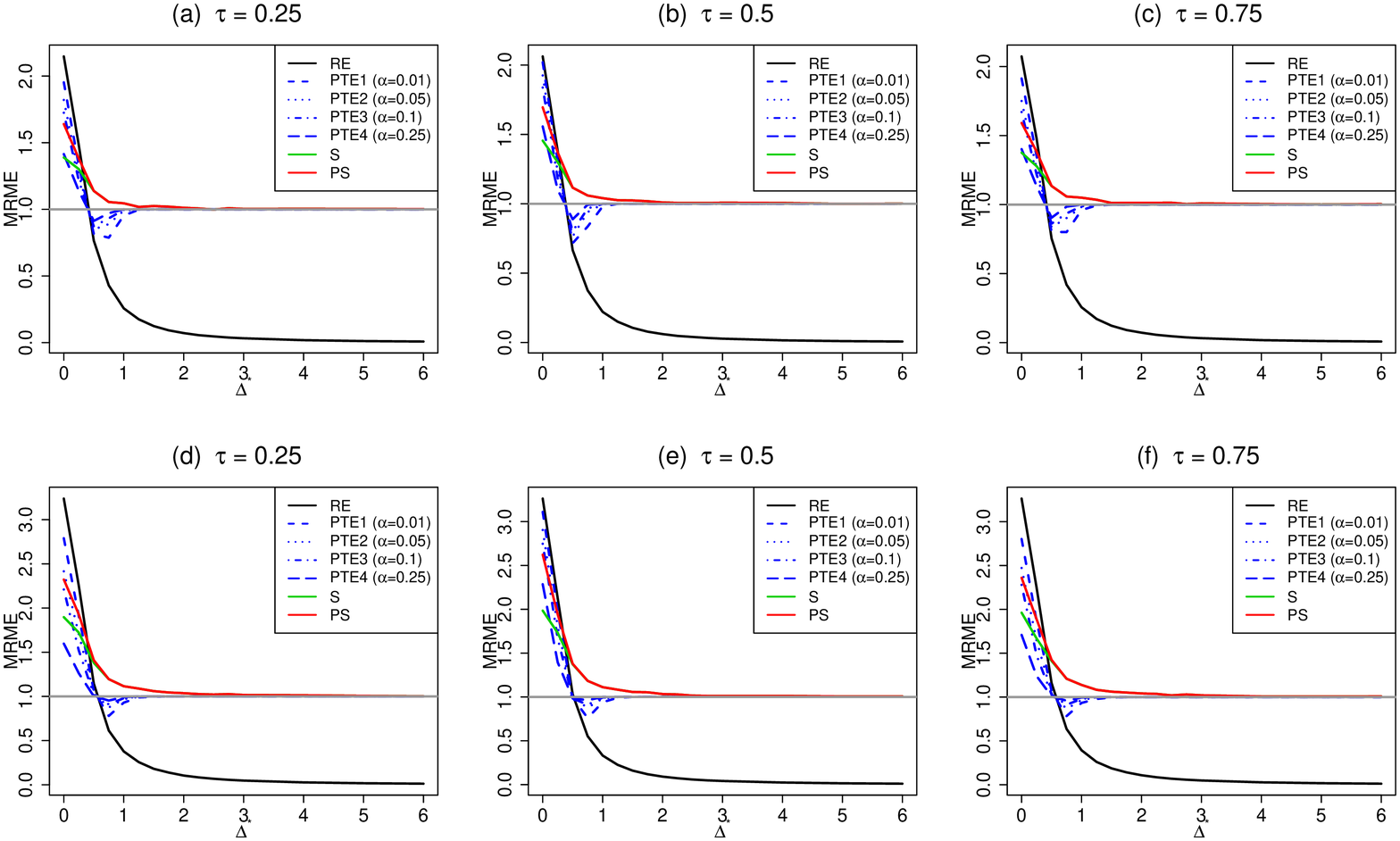}
   \caption{}
   \label{fig:Ng1} 
\end{subfigure}

\begin{subfigure}[b]{.9\textwidth}
   \includegraphics[height=8cm,width=14cm]{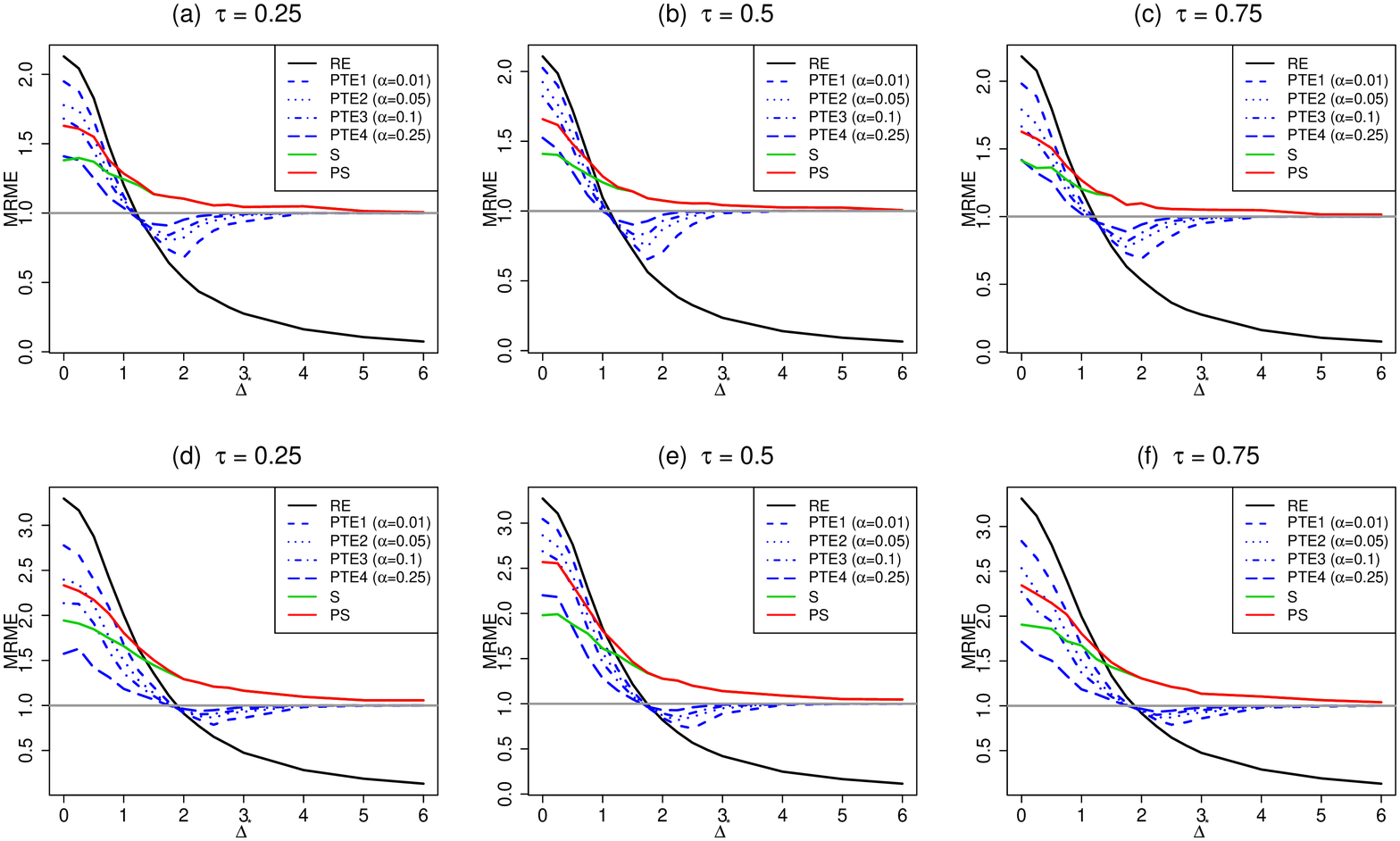}
     \caption{}
   \label{fig:Ng2}
\end{subfigure}
     \caption{$\sigma$ is 1 and 3 in the panels (A) and (B), respectively. Also, $p_1=5$ and $p_{\tau}=5$  in (a), (b) and (c), and $p_1=5$ and $p_{\tau}=10$ in (d), (e) and (f).}
\end{figure}

For each Monte Carlo dataset with 2500 replication, we compare the proposed $\hbbeta_{\tau}^{\rm FM}$ procedure using the relative median model error (MRME) which is defined as
\begin{equation}\label{eq:RME}
\textnormal{MRME}(\hbbeta_{\tau}^\ast;\hbbeta_{\tau}^{\rm FM}) = \frac{\textnormal{ME}(\hbbeta_{\tau}^{\rm FM})}{\textnormal{ME}(\hbbeta_{\tau}^{\ast})}
\end{equation}
If MRME of an estimator is larger than one, it is superior to $\hbbeta_{\tau}^{\rm FM}$.

\subsection{Performance Comparisons}
In this section, we illustrate two example to both investigate suggested estimators and compare their performance with listed quantile penalty estimations. Hence, we simulated data which contains a training dataset, validation set and an independent test set. In this section, the co-variates are scaled to have mean zero and unit variance. We fitted the models only using the training data and the tuning parameters were selected using the validation data. Finally, we computed the predictive mean absolute deviation (PMAD) criterion which is defined by
\begin{equation*}
\rm PMAD(\hbbeta_{\tau}^{\ast}) = \frac{1}{n_{test}}\sum_{i=1}^{n_{test}}\left | \by_{test}-\bX_{test}\hbbeta_{\tau}^{\ast} \right |.
\end{equation*}
We also use the notation $\cdot/\cdot/\cdot$ to describe the number of observations in the training, validation and test set respectively, e.g. $50/50/400$. Here are the details of the two examples:
\begin{itemize}
\item[1-] Each data set consists of $50/50/200$ observations. $\bbeta$ is specified by $\bbeta^{\top} = (3, 1.5, 0, 0, 2, 0, 0, 0)$ and $\sigma=3$. Also, we consider $\bX \sim N(0,\boldsymbol{\Sigma})$, where $\Sigma_{ij}=0.5^{|i-j|}$.

\item[2-] We keep all values the same as in (1) except
$\bbeta^{\top} = (3, 1.5, 0, 0, 2, \bold{0}_{10})$.
\end{itemize}

Table 1. shows the MAD values of pretest and shrinkage estimators for different distributions of $\varepsilon_i$ with given parameters at $25^{th}$, $50^{th}$ and $75^{th}$ quantiles, where $p_2=5$.

\begin{table}[ht]
\centering
\begin{adjustbox}{width=1\textwidth}
\begin{tabular}{rrcccccc}
  \toprule
$\tau$& &Normal& $\chi^2_5$ & $t_2$ & Laplace & Log-normal & Skew-Normal \\  
  \midrule
0.25&FM & 0.629(0.018) & 1.348(0.060) & 0.922(0.032) & 0.781(0.033) & 0.321(0.016) & 0.608(0.022) \\ 
&  SM & 0.176(0.009) & 0.361(0.021) & 0.246(0.015) & 0.230(0.010) & 0.097(0.007) & 0.162(0.008) \\ 
&  PT1 & 0.216(0.022) & 0.394(0.064) & 0.264(0.030) & 0.242(0.035) & 0.101(0.015) & 0.173(0.017) \\ 
&  PT2 & 0.225(0.024) & 0.430(0.075) & 0.292(0.039) & 0.267(0.039) & 0.105(0.017) & 0.177(0.025) \\ 
&  PT3 & 0.237(0.026) & 0.494(0.080) & 0.339(0.040) & 0.287(0.042) & 0.114(0.019) & 0.182(0.025) \\ 
&  PT4 & 0.268(0.028) & 0.497(0.082) & 0.560(0.043) & 0.360(0.045) & 0.124(0.020) & 0.206(0.030) \\ 
&  PS & 0.328(0.020) & 0.751(0.061) & 0.569(0.031) & 0.477(0.033) & 0.156(0.015) & 0.311(0.020) \\ 
\cmidrule(lr){3-8}
&  Ridge & 0.407(0.010) & 0.628(0.015) & 0.516(0.015) & 0.473(0.013) & 0.236(0.011) & 0.423(0.012) \\ 
&  Lasso & 0.261(0.010) & 0.452(0.015) & 0.371(0.015) & 0.320(0.013) & 0.135(0.008) & 0.254(0.011) \\ 
&  ENET & 0.256(0.010) & 0.433(0.015) & 0.354(0.015) & 0.293(0.013) & 0.133(0.008) & 0.247(0.011) \\ 
   \hline
0.5&  FM & 0.553(0.019) & 1.544(0.058) & 0.709(0.027) & 0.619(0.022) & 0.557(0.024) & 0.557(0.018) \\ 
&  SM & 0.160(0.009) & 0.496(0.026) & 0.199(0.011) & 0.141(0.009) & 0.153(0.01) & 0.160(0.008) \\ 
&  PT1 & 0.174(0.017) & 0.527(0.051) & 0.202(0.018) & 0.165(0.025) & 0.160(0.018) & 0.160(0.014) \\ 
&  PT2 & 0.176(0.018) & 0.529(0.059) & 0.203(0.024) & 0.169(0.027) & 0.176(0.019) & 0.174(0.022) \\ 
&  PT3 & 0.184(0.023) & 0.560(0.065) & 0.203(0.029) & 0.172(0.032) & 0.193(0.022) & 0.179(0.024) \\ 
&  PT4 & 0.256(0.028) & 0.631(0.074) & 0.259(0.036) & 0.202(0.035) & 0.222(0.028) & 0.215(0.026) \\ 
&  PS & 0.303(0.019) & 0.929(0.055) & 0.332(0.025) & 0.325(0.024) & 0.321(0.019) & 0.255(0.019) \\ 
\cmidrule(lr){3-8}
&  Ridge & 0.373(0.011) & 0.695(0.013) & 0.416(0.013) & 0.400(0.012) & 0.362(0.012) & 0.387(0.009) \\ 
& Lasso & 0.233(0.010) & 0.505(0.016) & 0.280(0.014) & 0.225(0.013) & 0.207(0.013) & 0.230(0.009) \\ 
&  ENET & 0.225(0.009) & 0.494(0.017) & 0.260(0.013) & 0.212(0.013) & 0.203(0.012) & 0.220(0.009) \\  
  \hline
0.75&  FM & 0.574(0.022) & 2.186(0.080) & 0.916(0.037) & 0.694(0.029) & 1.033(0.048) & 0.607(0.022) \\ 
&  SM & 0.174(0.009) & 0.656(0.032) & 0.250(0.015) & 0.201(0.012) & 0.314(0.018) & 0.159(0.009) \\ 
&  PT1 & 0.191(0.027) & 0.706(0.094) & 0.303(0.043) & 0.212(0.028) & 0.353(0.055) & 0.162(0.021) \\ 
&  PT2 & 0.206(0.029) & 0.777(0.105) & 0.330(0.048) & 0.224(0.032) & 0.534(0.057) & 0.175(0.025) \\ 
&  PT3 & 0.224(0.030) & 1.063(0.113) & 0.350(0.050) & 0.243(0.036) & 0.563(0.059) & 0.194(0.027) \\ 
&  PT4 & 0.262(0.031) & 1.756(0.111) & 0.478(0.053) & 0.513(0.040) & 0.792(0.056) & 0.214(0.031) \\ 
&  PS & 0.328(0.024) & 1.437(0.083) & 0.553(0.040) & 0.450(0.028) & 0.744(0.045) & 0.320(0.023) \\ 
\cmidrule(lr){3-8}
&  Ridge & 0.365(0.011) & 0.797(0.009) & 0.510(0.016) & 0.458(0.012) & 0.562(0.014) & 0.378(0.012) \\ 
&  Lasso & 0.248(0.011) & 0.627(0.016) & 0.356(0.016) & 0.298(0.013) & 0.404(0.016) & 0.236(0.009) \\ 
&  ENET & 0.240(0.011) & 0.603(0.017) & 0.331(0.016) & 0.292(0.013) & 0.391(0.016) & 0.217(0.009) \\
\hline
Mean&  OLS & 0.430(0.043) & 1.323(0.132) & 0.901(0.090) & 0.593(0.059) & 0.804(0.080) & 0.422(0.042) \\ 

\bottomrule
\end{tabular}
\end{adjustbox}
\caption{The simulated MAD values of estimators for Example 1. The values in parenthesis are present the standard errors of each methods
\label{Tab:Ex1}}
\end{table}

Furthermore, we consider that the errors follow one of the following distributions:
\begin{itemize}
\item[(i)] Standard normal distribution, 
\item[(ii)] Chi-square distribution with five degrees of freedom $(\chi^2_{5})$,
\item[(iii)] Student t distribution with two degrees of freedom $(t_2)$,
\item[(iv)] Laplace distribution with location parameter equals zero and scale parameter equals one, 
\item[(v)] Log-normal distribution with location parameter equals zero and scale parameter equals one and also
\item[(vi)] The random sample from the skew-normal distribution.
\end{itemize}

\begin{table}[ht]
\centering
\begin{adjustbox}{width=1\textwidth}
\begin{tabular}{rrcccccc}
  \toprule
$\tau$& &Normal& $\chi^2_5$ & $t_2$ & Laplace & Log-normal & Skew-Normal \\   
  \midrule
0.25&FM & 0.679(0.017) & 1.553(0.046) & 1.127(0.032) & 0.814(0.031) & 0.454(0.016) & 0.655(0.017) \\ 
&  SM & 0.099(0.004) & 0.219(0.011) & 0.152(0.009) & 0.128(0.007) & 0.049(0.004) & 0.100(0.004) \\ 
&  PT1 & 0.109(0.031) & 0.236(0.068) & 0.200(0.053) & 0.160(0.042) & 0.050(0.014) & 0.122(0.030) \\ 
&  PT2 & 0.120(0.033) & 0.271(0.077) & 0.240(0.055) & 0.195(0.046) & 0.057(0.019) & 0.125(0.033) \\ 
&  PT3 & 0.138(0.034) & 0.286(0.081) & 0.843(0.058) & 0.262(0.048) & 0.057(0.022) & 0.139(0.034) \\ 
&  PT4 & 0.546(0.034) & 0.342(0.086) & 0.996(0.056) & 0.655(0.048) & 0.060(0.024) & 0.506(0.035) \\ 
&  PS & 0.285(0.020) & 0.502(0.049) & 0.570(0.033) & 0.409(0.029) & 0.085(0.014) & 0.279(0.020) \\ 
\cmidrule(lr){3-8}
& Ridge & 0.399(0.005) & 0.433(0.002) & 0.433(0.003) & 0.432(0.005) & 0.300(0.008) & 0.398(0.005) \\ 
&  Lasso & 0.182(0.007) & 0.316(0.008) & 0.245(0.008) & 0.215(0.007) & 0.103(0.006) & 0.176(0.007) \\ 
&  ENET & 0.180(0.007) & 0.300(0.008) & 0.235(0.008) & 0.212(0.007) & 0.103(0.006) & 0.176(0.007) \\ 
   \hline
0.5&  FM & 0.616(0.016) & 1.812(0.052) & 0.787(0.029) & 0.717(0.025) & 0.651(0.023) & 0.603(0.017) \\ 
&  SM & 0.094(0.006) & 0.272(0.015) & 0.108(0.007) & 0.090(0.006) & 0.088(0.006) & 0.090(0.004) \\ 
&  PT1 & 0.100(0.020) & 0.283(0.057) & 0.110(0.024) & 0.105(0.028) & 0.095(0.026) & 0.092(0.018) \\ 
&  PT2 & 0.102(0.024) & 0.283(0.063) & 0.117(0.034) & 0.105(0.028) & 0.098(0.032) & 0.095(0.023) \\ 
&  PT3 & 0.104(0.025) & 0.285(0.069) & 0.117(0.039) & 0.107(0.030) & 0.100(0.033) & 0.100(0.027) \\ 
&  PT4 & 0.119(0.028) & 0.334(0.082) & 0.126(0.044) & 0.113(0.034) & 0.115(0.035) & 0.104(0.029) \\ 
&  PS & 0.149(0.016) & 0.402(0.048) & 0.177(0.024) & 0.138(0.020) & 0.174(0.021) & 0.132(0.017) \\ 
\cmidrule(lr){3-8}
&  Ridge & 0.386(0.005) & 0.433(0.001) & 0.414(0.005) & 0.400(0.006) & 0.377(0.006) & 0.370(0.005) \\ 
&  Lasso & 0.175(0.007) & 0.358(0.008) & 0.200(0.007) & 0.192(0.007) & 0.167(0.007) & 0.164(0.006) \\ 
&  ENET & 0.168(0.007) & 0.349(0.008) & 0.195(0.007) & 0.191(0.007) & 0.164(0.007) & 0.161(0.006) \\ 
  \hline
0.75&  FM & 0.640(0.018) & 2.362(0.064) & 1.079(0.033) & 0.847(0.028) & 1.296(0.044) & 0.679(0.022) \\ 
&  SM & 0.105(0.006) & 0.366(0.017) & 0.131(0.008) & 0.119(0.006) & 0.198(0.010) & 0.099(0.004) \\ 
&  PT1 & 0.122(0.023) & 0.585(0.121) & 0.194(0.054) & 0.163(0.039) & 0.848(0.065) & 0.113(0.031) \\ 
&  PT2 & 0.139(0.029) & 1.890(0.123) & 0.239(0.056) & 0.197(0.045) & 0.970(0.064) & 0.140(0.035) \\ 
&  PT3 & 0.165(0.032) & 2.033(0.119) & 0.464(0.056) & 0.224(0.046) & 1.143(0.060) & 0.147(0.035) \\ 
&  PT4 & 0.438(0.033) & 2.152(0.113) & 0.917(0.056) & 0.677(0.046) & 1.215(0.054) & 0.395(0.036) \\ 
&  PS & 0.304(0.018) & 1.436(0.072) & 0.475(0.036) & 0.426(0.027) & 0.818(0.037) & 0.271(0.023) \\ 
\cmidrule(lr){3-8}
&  Ridge & 0.396(0.006) & 0.433(0.000) & 0.433(0.004) & 0.430(0.005) & 0.433(0.002) & 0.400(0.005) \\ 
&  Lasso & 0.178(0.007) & 0.416(0.008) & 0.238(0.009) & 0.216(0.009) & 0.279(0.009) & 0.184(0.007) \\ 
&  ENET & 0.173(0.007) & 0.406(0.008) & 0.233(0.009) & 0.213(0.009) & 0.277(0.009) & 0.182(0.007) \\ 
  \hline
Mean&  OLS & 0.519(0.052) & 1.556(0.156) & 1.177(0.118) & 0.671(0.067) & 0.927(0.093) & 0.489(0.049) \\ 

\bottomrule
\end{tabular}
\end{adjustbox}
\caption{The simulated MAD values of estimators for Example 2. The values in parenthesis are present the standard errors of each methods.}
\label{Tab:Ex2}
\end{table}

We report the result of examples in Tables \ref{Tab:Ex1} and \ref{Tab:Ex2} as follows:
\begin{itemize}
\item[(i)] For $\tau=0.25$; PT1 outperforms the shrinkage estimators. \item[(ii)] For $\tau=0.5$, PT1 more efficient than other estimators except $\chi^2_{5}$ distribution also when $p_2$ value increases, all pretest estimators give better results than shrinkage estimators.
\item[(iii)] For $\tau=0.75$, ENET gives a better result even though the performance of PT1  expected to increase when $p_2$ increased.
\item[(iv)] For $\tau=0.25, 0.5, 0.75$; in $\chi^2_{5}$ distribution, the MAD values of suggested estimators are larger than the MAD values of listed distributions.
\end{itemize}

\clearpage
\section{Real Data Application}
\label{RAA}

\subsection{Prostate Data}
Prostate data came from the study of \cite{Prostate data} about correlation between the level of prostate specific antigen (PSA), and a number of clinical measures in men who were about to receive radical prostatectomy. The data consist of 97 measurements on the following variables of in Table \ref{Tab:variables:prostate}: 

\begin{table}[!htbp]
\small
\centering
\begin{tabular}{ll}
\toprule
\textbf{Variables} & \textbf{Descriptions} \\
\midrule

\textbf{Dependent Variable} &\\
lpsa & Log of prostate specific antigen (PSA) \\
\midrule

\textbf{Covariates}   \\
lcavol   & Log cancer volume  \\
lweight  & Log prostate weight \\
age      &  Age in years\\
lbph     & Log of benign prostatic hyperplasia amount  \\
svi      & Seminal vesicle invasion  \\
lcp      & Log of capsular penetration  \\
gleason  &  Gleason score\\
pgg45    & Percent of Gleason scores 4 or 5 \\
\bottomrule
\end{tabular}
\caption{Lists and Descriptions of Variables for Prostate data set
\label{Tab:variables:prostate}}
\end{table}

\subsection{Barro Data}
The Barro data came from the study of \cite{Barro data} about distribution of education attainment and human capital by genders and by 5-year age intervals in 138 countries from 1965 to 1985. The \textit{quantreg} version of the Barro Growth Data used in \cite{Barro-Growth}. This is a regression data set consisting of 161 observations on determinants of cross country GDP growth rates. There are 13 covariates and a description of  the variables in the Barro dataset is given in Table \ref{Tab:variables:Barro}. The goal is to study the effect of the covariates on GDP.

\begin{table}[!htbp]
\small
\centering
\begin{tabular}{ll}
\toprule
\textbf{Variables} & \textbf{Descriptions} \\
\midrule

\textbf{Dependent Variable} &\\
GDP & Annual Change Per Capita \\
\midrule

\textbf{Covariates}   \\
lgdp2   & Initial 
Per Capita GDP  \\
mse2  & Male Secondary Education \\
fse2      &  Female Secondary Education\\
fhe2     & Female Higher Education  \\
mhe2      & Male Higher Education  \\
lexp2      & Life Expectancy  \\
lintr2  &  Human Capital\\
gedy2    & Education/GDP\\
Iy2      & Investment/GDP  \\
gcony2  &  Public Consumption/GDP\\
lblakp2    & Black Market Premium\\
pol2  &  Political Instability\\
ttrad2    & Growth Rate Terms Trade\\
\bottomrule
\end{tabular}
\caption{Lists and Descriptions of Variables for Barro data set
\label{Tab:variables:Barro}}
\end{table}

The results of prostate data application may be summarized as follows:
\begin{itemize}
\item[(i)] Ridge estimator is far more effective than LSE and the other regression estimators as expected because of outliers and multicollinearity exist in the prostate data at the same time.
\item[(ii)] The results demonstrate that after ridge estimator ENET gives better results than LSE only at lower 20th quantile and after exceeding in 70th quantile.
\item[(iii)]  For 40th to 80th quantile all of the proposed estimators are more efficient than LSE.
\end{itemize}
The results of barro data application may be summarized as follows:
\begin{itemize}
\item[(i)] Except for the quantiles up to the 20th and after 80th MAD value of PT1 is the lowest so it is outperformed other estimators including LSE. 
\item[(ii)] While become distant from 50th quantile as the MAD value of the PT1 increases its efficiency decreased.
\end{itemize}

We plot some diagnostics for residuals of fit models of both data sets in Figure \ref{Outliers_datas}. According to the Figure \ref{fig:prostate:res}, one may suspect that the observation 39, 69, 95 may be outliers. Also, the results of apply outlierTest function in the car package in R confirm that observations 39 is an outlier. Furthermore, we calculate the ratio of largest eigenvalue to smallest eigenvalue of design matrix of prostate data is approximately $243.302$ which indicates that there multicollinearity in the data set.  According to the Figure \ref{fig:barro:cor}, one may suspect that the observation 74, 116, 120 may be outliers. Again, outlierTest function confirm that observations 74 is an outlier. Furthermore, we calculate the ratio of largest eigenvalue to smallest eigenvalue of design matrix of prostate data is approximately $826.965$ which again indicates that there multicollinearity in the data set.

\begin{figure}
    \centering
    \begin{subfigure}[b]{1\textwidth}
        \centering
        \includegraphics[width=12cm,height=4cm]{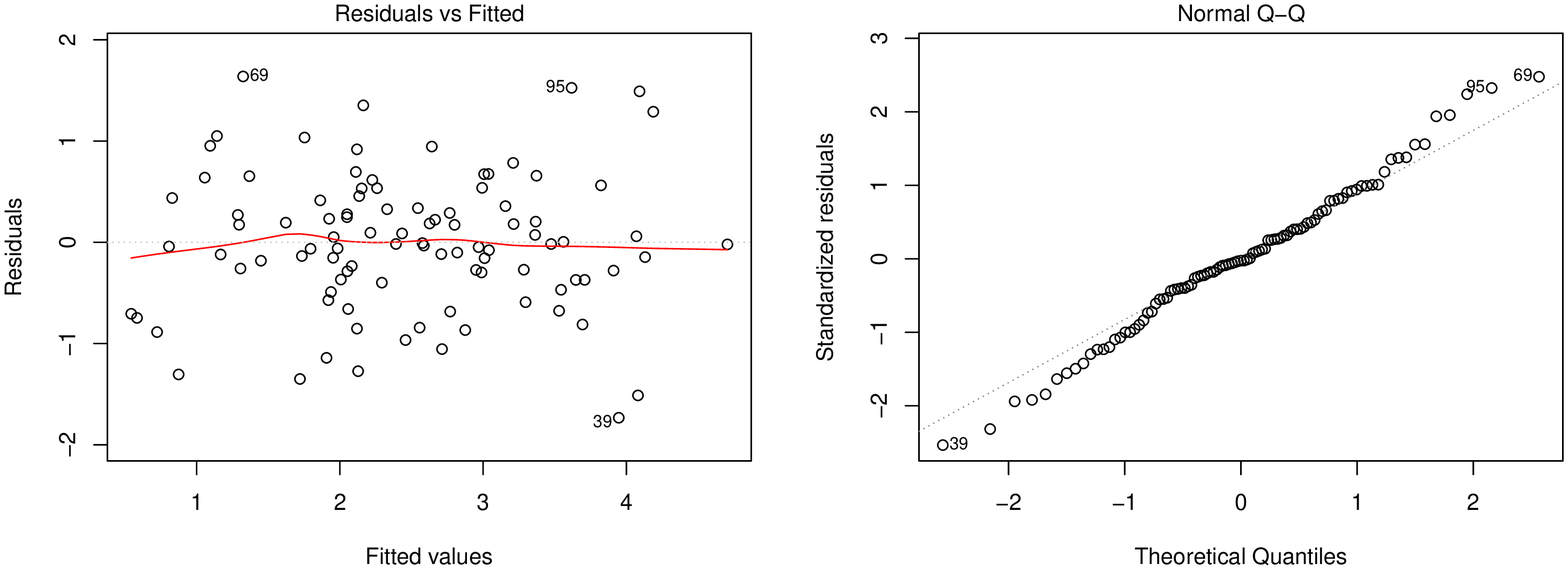}
        \caption{Prostate Data}
        \label{fig:prostate:res}
    \end{subfigure}
    \\
    \begin{subfigure}[b]{1\textwidth}
        \centering
        \includegraphics[width=12cm,height=4cm]{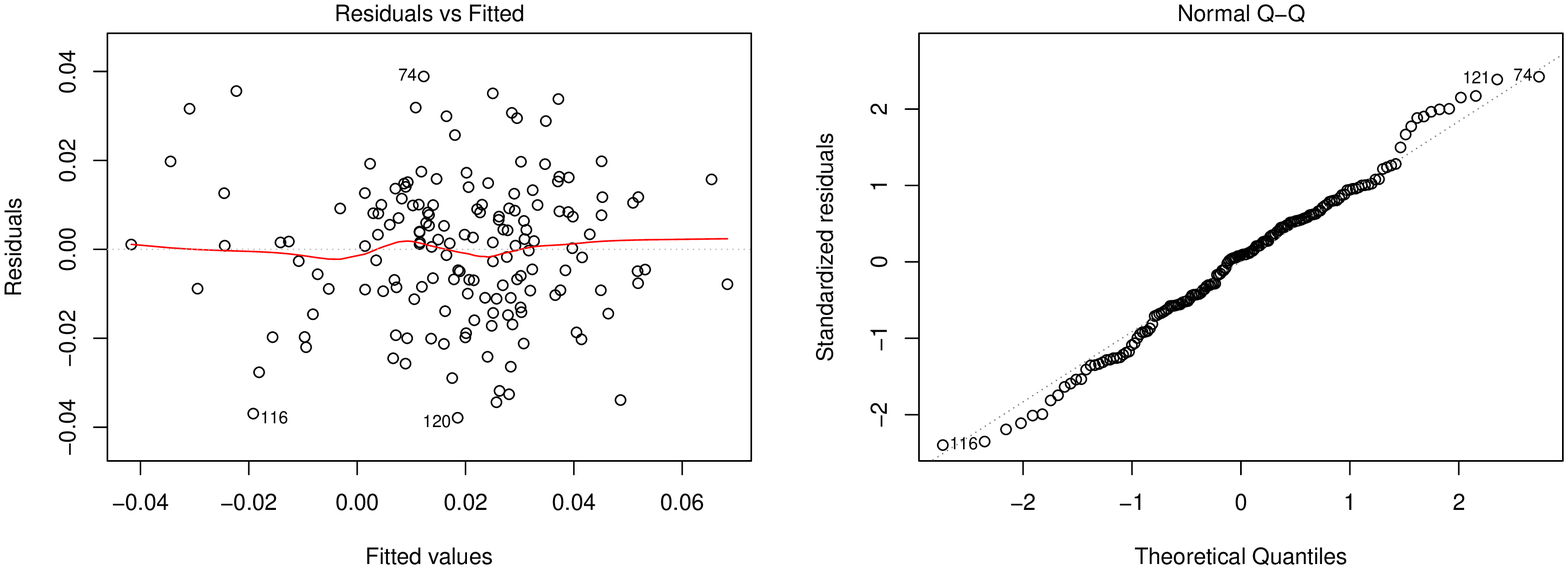}
        \caption{Barro Data}
        \label{fig:barro:cor}
    \end{subfigure}
    \caption{Outlier Plots}
    \label{Outliers_datas}
\end{figure}


\begin{table}[ht]
\begin{adjustbox}{width=1\textwidth}
\centering
\begin{tabular}{rcccccccccc}
\toprule
Data&$\tau$ & 0.1 & 0.2 & 0.3 & 0.4 & 0.5 & 0.6 & 0.7 & 0.8 & 0.9 \\ 
\midrule

Prostate&FM & 5.821(0.083) & 5.520(0.068) & 5.170(0.058) & 4.928(0.057) & 4.723(0.057) & 4.476(0.053) & 4.449(0.057) & 4.900(0.060) & 5.388(0.077) \\ 
&  SM & 3.766(0.061) & 3.379(0.063) & 3.047(0.053) & 2.945(0.045) & 2.956(0.038) & 2.991(0.041) & 3.135(0.044) & 3.244(0.048) & 3.466(0.054) \\ 
&  PT1 & 5.769(0.088) & 4.757(0.078) & 4.172(0.079) & 3.725(0.072) & 3.555(0.065) & 3.374(0.057) & 3.609(0.060) & 3.865(0.063) & 5.095(0.085) \\ 
&  PT2 & 5.789(0.087) & 4.993(0.076) & 4.584(0.078) & 3.997(0.073) & 3.838(0.068) & 3.707(0.061) & 3.735(0.061) & 4.044(0.064) & 5.243(0.084) \\ 
&  PT3 & 5.810(0.086) & 5.082(0.076) & 4.794(0.076) & 4.310(0.069) & 3.986(0.067) & 3.914(0.060) & 3.909(0.063) & 4.110(0.067) & 5.286(0.081) \\ 
&  PT4 & 5.813(0.084) & 5.315(0.074) & 4.986(0.07) & 4.656(0.067) & 4.313(0.063) & 4.045(0.059) & 4.071(0.063) & 4.275(0.067) & 5.346(0.079) \\ 
&  PS & 5.415(0.082) & 4.576(0.063) & 4.216(0.061) & 3.905(0.055) & 3.827(0.050) & 3.640(0.048) & 3.663(0.052) & 3.981(0.053) & 4.908(0.072) \\ 
\cmidrule(lr){3-11}
&  Ridge & 2.035(0.027) & 2.146(0.027) & 2.407(0.027) & 2.529(0.023) & 2.661(0.022) & 2.582(0.023) & 2.471(0.024) & 2.296(0.027) & 2.032(0.026) \\ 
&  Lasso & 4.860(0.081) & 4.403(0.067) & 4.480(0.065) & 4.369(0.058) & 4.434(0.055) & 4.036(0.050) & 3.823(0.054) & 3.817(0.058) & 4.047(0.074) \\ 
&  ENET & 4.760(0.082) & 4.185(0.068) & 4.323(0.064) & 4.165(0.059) & 4.427(0.055) & 3.989(0.050) & 3.679(0.053) & 3.681(0.056) & 3.820(0.074) \\ 
\cmidrule(lr){2-11}
&   &Mean \\
\cmidrule(lr){2-11}
&   LSE&4.469(0.283)\\

   \cmidrule(lr){2-11}
   &$\tau$ & 0.1 & 0.2 & 0.3 & 0.4 & 0.5 & 0.6 & 0.7 & 0.8 & 0.9 \\ 
   \cmidrule(lr){2-11}
   &&$\rm MAD\times 10^2$&$\rm MAD\times 10^2$&$\rm MAD\times 10^2$&$\rm MAD\times 10^2$&$\rm MAD\times 10^2$&$\rm MAD\times 10^2$&$\rm MAD\times 10^2$&$\rm MAD\times 10^2$&$\rm MAD\times 10^2$\\ 
\midrule
Barro  &FM & 3.872(0.021) & 3.238(0.022) & 2.916(0.021) & 2.738(0.023) & 2.766(0.023) & 2.816(0.023) & 2.988(0.026) & 3.298(0.027) & 3.755(0.027) \\ 
&  SM & 3.444(0.012) & 2.689(0.009) & 2.315(0.006) & 2.093(0.006) & 2.034(0.005) & 2.125(0.006) & 2.341(0.007) & 2.743(0.008) & 3.377(0.010) \\ 
&  PT1 & 3.763(0.022) & 2.825(0.024) & 2.350(0.023) & 2.169(0.026) & 2.072(0.027) & 2.188(0.028) & 2.436(0.029) & 2.882(0.031) & 3.699(0.028) \\ 
&  PT2 & 3.826(0.022) & 2.950(0.026) & 2.489(0.026) & 2.227(0.028) & 2.187(0.029) & 2.371(0.029) & 2.501(0.031) & 3.043(0.030) & 3.728(0.028) \\ 
&  PT3 & 3.834(0.022) & 3.009(0.025) & 2.616(0.026) & 2.367(0.029) & 2.450(0.029) & 2.558(0.028) & 2.661(0.030) & 3.130(0.031) & 3.743(0.028) \\ 
&  PT4 & 3.865(0.022) & 3.094(0.025) & 2.761(0.025) & 2.606(0.027) & 2.576(0.028) & 2.677(0.028) & 2.828(0.029) & 3.213(0.030) & 3.750(0.028) \\ 
&  PS & 3.709(0.018) & 2.896(0.018) & 2.530(0.017) & 2.330(0.019) & 2.308(0.019) & 2.395(0.020) & 2.566(0.023) & 2.976(0.024) & 3.616(0.025) \\ 
\cmidrule(lr){3-11}
&  Ridge & 3.441(0.019) & 2.876(0.019) & 2.670(0.021) & 2.631(0.022) & 2.731(0.022) & 2.688(0.021) & 2.748(0.025) & 2.998(0.024) & 3.386(0.025) \\ 
&  Lasso & 3.406(0.019) & 2.892(0.019) & 2.648(0.021) & 2.618(0.022) & 2.717(0.022) & 2.668(0.021) & 2.725(0.024) & 2.952(0.024) & 3.371(0.025) \\ 
&  ENET & 3.380(0.020) & 2.860(0.019) & 2.633(0.021) & 2.616(0.022) & 2.733(0.022) & 2.663(0.020) & 2.735(0.025) & 2.941(0.025) & 3.333(0.025) \\ 
\cmidrule(lr){2-11}
&   &Mean \\   
\cmidrule(lr){2-11}
&   LSE&2.672(0.169)\\

\bottomrule
\end{tabular}
\end{adjustbox}
\caption{The MAD values and standard errors (in parenthesis) of each methods for Prostate and Barro Data}
\end{table}


\section{Conclusions}
\label{conc}

In this paper, we consider pretest and shrinkage estimation strategies based on quantile regression which is a good alternative when the error distribution does not provide the LSE assumptions. To this end, we combined both under-fitted and over-fitted estimation in a optimal way. Further, a Monte Carlo simulation study is conducted to investigate the performance of the suggested estimators in two aspects: one is for investigation of asymptotic properties of listed estimators and the other is for a comparative study with quantile-type estimators, namely Lasso, Ridge and ENET, when the errors follow the Chi-Square, Student t, Laplace, Log-Normal and Skew-Normal distribution. According to first simulation results, the performance of sub-model is the best when the null hypothesis is true. On the other hand, it loses its efficiency and goes to zero when we violate the null hypothesis. The performance of pretest estimations are better than Stein--type estimation when the null hypothesis is true while it loses its efficiency even worse than the full model estimation when we violate with small degree of null hypothesis, and it acts like the full model estimation when the alternative hypothesis is exactly true. Regarding Shrinkage estimation, the PS estimation is always superior to the shrinkage estimator. The performance of PS loses its efficiency when we violate the null hypothesis, however, its performance is superior the full model estimation even the null hypothesis is not true. According to second simulation results, we demonstrated that the listed estimators performs better than quantile $L_1$-type estimators and the LSE when we look different percent quantiles. Also, we applied two real data examples both of which has the problem of multicollinearity and they have outliers. According to the real data example results, our suggested estimators outshine the LSE and the quantile-type estimators with some minor exceptions. Hence, it is safely concluded that the pretest and shrinkage estimation strategies are more efficient than the LSE and quantile-type estimators when the number of nuisance parameter is large. Finally, we obtained asymptotic distributions of the listed estimators. Our asymptotic theory is well supported by numerical analysis.

\section*{Appendix}
\begin{lemma} \label{lem_JB}
Let $\bX$ be $q-$dimensional normal vector distributed as $%
N\left( \boldsymbol{\mu }_{x},\boldsymbol{\Sigma }%
_{q}\right) ,$ then, for a measurable function of of $\varphi ,$ we have
\begin{align*}
\mathbb{E}\left[ \bX\varphi \left( \bX^{\top}\bX%
\right) \right] =&\boldsymbol{\mu }_{x}\mathbb{E}\left[ \varphi \chi _{q+2}^{2}\left(
\Delta \right) \right] \\
\mathbb{E}\left[ \boldsymbol{XX}^{\top}\varphi \left( \bX^{\top}%
\bX\right) \right] =&\boldsymbol{\Sigma }_{q}\mathbb{E}\left[ \varphi
\chi _{q+2}^{2}\left( \Delta \right) \right] +\boldsymbol{\mu }_{x}\boldsymbol{\mu }_{x}^{\top}\mathbb{E}\left[ \varphi \chi
_{q+4}^{2}\left( \Delta \right) \right]
\end{align*}
where $\chi_{v}^{2}\left( \Delta \right)$ is a non-central chi-square distribution with $v$ degrees of freedom and non-centrality parameter $\Delta$.
\end{lemma}
\begin{proof} It can be found in \cite{judge-bock1978} \end{proof}
\begin{proof}[Proof of Proposition \ref{prop_vector_dist}]
Using the definition of asymptotic bias, we have
\begin{eqnarray}\label{bias_FM}
\mathcal{B}\left(\hbbeta_{1,\tau}^{\rm FM}\right)=\mathbb{E}\left\{\underset{%
n\rightarrow \infty }{\lim }\sqrt{n}\left( \hbbeta_{1,\tau}^{\rm FM}-\bbeta_{1}\right)\right\}
&=&\mathbf{0}_{p_1}
\end{eqnarray}
and also using the definition of asymptotic covariance, we get
\begin{eqnarray}\label{cov_FM}
\boldsymbol{\Gamma} \left( \FM \right) &=&\mathbb{E}\left\{
\underset{n\rightarrow \infty }{\lim }{n}\left( \hbbeta_{1}^{\rm FM} -\bbeta_{1,\tau}\right)\left( \hbbeta_{1,\tau}^{\rm FM} -\bbeta_{1,\tau}\right) ^{'}\right\} \no\\
&=& Cov\left( \bvt _{1},\bvt _{1}^{'}\right)+\mathbb{E}\left( \bvt _{1}\right) \mathbb{E}\left( \bvt _{1}^{'}\right)\no\\
&=& Cov\left(\bvt_1,\bvt_1^{'}\right) \no\\
&=&w^2\bD_{11.2}^{-1}
\end{eqnarray}
where $\bD_{11.2}^{-1}=(\bD_{11}-\bD_{12}\bD_{22}^{-1}\bD_{21})^{-1}$. \\
Thus, $\bvt _{1}  \sim\mathcal{N} \left(\boldsymbol{0 }_{p_1}, w^2\bD_{11.2}^{-1} \right).$ \\
Similarly, we obtain
\begin{eqnarray}\label{bias_SM}
\mathcal{B}\left( \hbbeta_{1,\tau}^{\rm SM}\right)
&=&\mathbb{E}\left\{ \underset{n\rightarrow \infty }{\lim }\sqrt{n}\left( \boldsymbol{%
\widehat{\beta}}_{1,\tau}^{\rm SM} -\bbeta_{1,\tau}\right) \right\}\no \\
&=&\mathbb{E}\left\{ \underset{n\rightarrow \infty }{\lim }\sqrt{n}\left( \boldsymbol{%
\widehat{\beta}}_{1,\tau}^{\rm FM} +\bD_{11}^{-1}\bD_{12}\boldsymbol{%
\widehat{\beta}}_{2,\tau}^{\rm FM}-\bbeta_{1,\tau}\right) \right\}\no \\
&=&\mathbb{E}\left\{ \underset{n\rightarrow \infty }{\lim }\sqrt{n}\left( \boldsymbol{%
\widehat{\beta}}_{1,\tau}^{\rm FM} -\bbeta_{1,\tau}\right) \right\}+\mathbb{E}\left\{ \underset{n\rightarrow \infty }{\lim }\sqrt{n}\left( \bD_{11}^{-1}\bD_{12}\boldsymbol{%
\widehat{\beta}}_{2,\tau}^{\rm FM}\right) \right\}\no\\
&=&\bD_{11}^{-1}\bD_{12}\kappa \no\\
&=&\bdelta
\end{eqnarray} 
and using the fact that 
$Cov\left(\bvt_2 \bvt_2'\right)=\omega^2\bD_{11}^{-1}$,
the covariance of $\vartheta_2$ is computed as follows;
\begin{eqnarray}\label{cov_SM}
\boldsymbol{\Gamma} \left( \hbbeta_{1,\tau}^{\rm SM} \right) &=&\mathbb{E}\left\{
\underset{n\rightarrow \infty }{\lim }{n}\left( \hbbeta%
_{1,\tau}^{\rm SM} -\bbeta_{1,\tau}\right)\left( \hbbeta%
_{1,\tau}^{\rm SM} -\bbeta_{1,\tau}\right) ^{'}\right\} \no\\
&=& \Cov\left( \bvt_{2}\bvt_{2}^{'}\right) +\mathbb{E}\left( \bvt_{2}\right) \mathbb{E}\left( \bvt_{2}^{'}\right)\no \\
&=&\omega^2\bD_{11}^{-1}+\bdelta\bdelta^{'}.
\end{eqnarray}
Thus, $\bvt _{2}  \sim\mathcal{N} \left(\boldsymbol{\delta },\omega^2 \bD_{11}^{-1}\right)$.
We know that
$\hbbeta_{1,\tau}^{\rm SM}=\hbbeta_{1,\tau}^{\rm FM}+\bD_{11}^{-1}\bD_{12}\hbbeta_{2,\tau}^{\rm FM}
$, so we have
\begin{eqnarray*}
\mathbb{E}\left(\bvt_3\right)
&=&\mathbb{E}\left\{\underset{%
n\rightarrow \infty }{\lim }\sqrt{n}\left(\hbbeta_{1,\tau}^{\rm FM}-\hbbeta_{1,\tau}^{\rm SM}\right)\right\} \\ 
&=& \mathbb{E}\left\{ \underset{%
n\rightarrow \infty }{\lim }-\sqrt{n}(\bD_{11}^{-1}\bD_{12}\hbbeta_{2,\tau}^{\rm FM}) \right\}  \\
&=&-\bdelta
\end{eqnarray*}
and we compute the covariance matrix of $\bvt_3$ as follows:
\begin{eqnarray*}
Cov\left(\bvt_3,\bvt_3'\right)&=&Cov\left(\hbbeta_{1,\tau}^{\rm FM}-\hbbeta_{1,\tau}^{\rm SM}\right) \\
&=&\bD_{11}^{-1}\bD_{12}Var\left(\hbbeta_{2,\tau}^{\rm FM}\right)\bD_{21}\bD_{11}^{-1}\\
&=&\omega^2\bD_{11}^{-1}\bD_{12}\bD_{22.1}^{-1}\bD_{21}\bD_{11}^{-1}\\
&=&\bPhi
\end{eqnarray*}
where $\bD_{22.1}=\bD_{22}-\bD_{21}\bD_{11}^{-1}\bD_{12}$.
Thus, $\bvt_3\sim{\mathcal{N}\left(-\bdelta,\bPhi\right)}$. 
Now, we also need to compute $Cov\left(\bvt_1,\bvt_3\right)$ and $Cov\left(\bvt_2,\bvt_3\right)$. Since
\begin{eqnarray*}
Cov\left(\hbbeta_{1,\tau}^{\rm FM},\hbbeta_{1,\tau}^{\rm SM}\right)&=&Cov\left(\hbbeta_{1,\tau}^{\rm FM},\hbbeta_{1,\tau}^{\rm FM}+\bD_{11}^{-1}\bD_{12}\hbbeta_{2,\tau}^{\rm FM}\right) \\
&=& Cov\left(\hbbeta_{1,\tau}^{\rm FM},\hbbeta_{1,\tau}^{\rm FM}\right)+Cov\left(\hbbeta_{1,\tau}^{\rm FM},\bD_{11}^{-1}\bD_{12}\hbbeta_{2,\tau}^{\rm FM}\right)\\
&=&\omega^2 \bD_{11.2}^{-1}+\omega^2\bD_{12}\bD_{21}\bD_{11}^{-1}
\end{eqnarray*}
We have the following
\begin{eqnarray*}
Cov\left( \bvt_1,\bvt_3 \right) &=& Cov\left(\hbbeta_{1,\tau}^{\rm FM},\hbbeta_{1,\tau}^{\rm FM}-\hbbeta_{1,\tau}^{\rm SM}\right) \\
&=& Cov\left(\hbbeta_{1,\tau}^{\rm FM},\hbbeta_{1,\tau}^{\rm FM}\right)-Cov\left(\FM,\SM\right) \\
&=&-\omega^2\bD_{12}\bD_{21}\bD_{11}^{-1}\\
&=&\bSigma_{12}
\end{eqnarray*}
Thus, we have
\begin{eqnarray*}
\left(\begin{array}{c}\bvt_1\\ \bvt_3\end{array}\right)\sim{N\left(\left[\begin{array}{c}\boldsymbol{0 }_{p_1}\\ -\bdelta\end{array}\right], \begin{bmatrix}\omega^2 \bD_{11.2}^{-1} & \bSigma_{12} \\\bSigma_{21} & \bPhi \end{bmatrix}\right)}
\end{eqnarray*}
$Cov\left(\bvt_2,\bvt_3\right)$ is obtained as follows:
\begin{eqnarray*}
Cov\left(\bvt_2,\bvt_3\right)&=&Cov\left(\hbbeta_{1,\tau}^{\rm SM},\hbbeta_{1,\tau}^{\rm FM}-\hbbeta_{1,\tau}^{\rm SM}\right) \\
&=&Cov\left(\hbbeta_{1,\tau}^{\rm SM},\hbbeta_{1,\tau}^{\rm FM}\right)-Cov\left(\hbbeta_{1,\tau}^{\rm SM},\hbbeta_{1,\tau}^{\rm SM}\right) \\
&=&\omega^2 \bD_{11.2}^{-1}+\omega^2\bD_{12}\bD_{21}\bD_{11}^{-1}-\omega^2\bD_{11}\\
&=& \bSigma^*
\end{eqnarray*}
Thus, the result follows as
$$\left(
\begin{array}{c}
\bvt_{3} \\
\bvt_{2}%
\end{array}%
\right) \sim\mathcal{N}\left[ \left(
\begin{array}{c}
-\boldsymbol{\delta } \\
\boldsymbol{\delta }%
\end{array}%
\right) ,\left(
\begin{array}{cc}
\bPhi & \bSigma^* \\
\bSigma^* & \omega^2 \bD_{11}%
\end{array}%
\right) \right]$$
\end{proof}
\begin{proof} [Proof of Theorem \ref{bias}]
The asymptotic bias of $\FM$ and $\SM$ are respectively given in Equations (\ref{bias_FM}) and (\ref{bias_FM}). Now, we continue with computing the bias of $\PT$ as follows:
\begin{eqnarray}\label{bias_PT}
\mathcal{B}\left(\hbbeta_{1,\tau}^{\rm PT}\right)&=&\mathbb{E}\left\{ \underset{n\rightarrow \infty }{\lim }\sqrt{n}\left( \boldsymbol{%
\widehat{\beta}}_{1,\tau}^{\rm PT} -\bbeta_{1,\tau}\right) \right\} \no\\
&=&\mathbb{E}\left\{ \underset{n\rightarrow \infty }{\lim }\sqrt{n}\left( \boldsymbol{%
\widehat{\beta}}_{1,\tau}^{\rm FM} -\left(\boldsymbol{%
\widehat{\beta}}_{1,\tau}^{\rm FM}-\boldsymbol{%
\widehat{\beta}}_{1,\tau}^{\rm SM}\right)\textrm{I}\left(\mathcal{W}_{n}< c_{n,\alpha}\right)-\bbeta_{1,\tau}\right) \right\}\no\\
&=&\mathbb{E}\left\{ \underset{n\rightarrow \infty }{\lim }\sqrt{n}\left( \boldsymbol{%
\widehat{\beta}}_{1,\tau}^{\rm FM} -\bbeta_{1,\tau}\right) \right\} \no\\
&&-\mathbb{E}\left\{ \underset{n\rightarrow \infty }{\lim }\sqrt{n}\left( \boldsymbol{%
\widehat{\beta}}_{1,\tau}^{\rm FM} -\boldsymbol{%
\widehat{\beta}}_{1,\tau}^{\rm SM}\right)\textrm{I}\left(\mathcal{W}_{n}<\chi^2_{p_2,\alpha}\right) \right\}\no\\
&=&\boldsymbol{\delta }\mathbb{H}_{p_{2}+2}\left( \chi^2_{p_2,\alpha};\Delta \right)
\end{eqnarray}
the last step is due to \cite{judge-bock1978}. The asymptotic bias of $\SS$ is computed as
\begin{eqnarray}\label{bias_S}
\mathcal{B}\left( \hbbeta_{1,\tau}^{\rm S}\right)&=&\mathbb{E}\left\{ \underset{n\rightarrow \infty }{\lim }\sqrt{n}\left( \boldsymbol{%
\widehat{\beta}}_{1,\tau}^{\rm S} -\bbeta_{1,\tau}\right) \right\} \no\\
&=&\mathbb{E}\left\{ \underset{n\rightarrow \infty }{\lim }\sqrt{n}\left( \boldsymbol{%
\widehat{\beta}}_{1,\tau}^{\rm FM} -\left(\boldsymbol{%
\widehat{\beta}}_{1,\tau}^{\rm FM}-\boldsymbol{%
\widehat{\beta}}_{1,\tau}^{\rm SM}\right)(p_2-2)\mathcal{W}_{n}^{-1}-\bbeta_{1,\tau}\right) \right\}\no\\
&=&\mathbb{E}\left\{ \underset{n\rightarrow \infty }{\lim }\sqrt{n}\left( \boldsymbol{%
\widehat{\beta}}_{1,\tau}^{\rm FM} -\bbeta_{1,\tau}\right) \right\}\no\\
&&-\mathbb{E}\left\{ \underset{n\rightarrow \infty }{\lim }\sqrt{n}\left( \FM -\SM \right)(p_{2}-2)\mathcal{W}_{n}^{-1} \right\}\no\\
&=&d\bdelta \mathbb{E}\left\{\chi _{p_{2}+2
}^{-2}\left( \Delta^2\right)\right\}.
\end{eqnarray}
Finally, the asymptotic bias of $\PS$ is obtained as follows
\begin{eqnarray*}\label{bias_PS}
\lefteqn{\mathcal{B}\left( \PS \right)}\\
&=&\mathbb{E}\left\{ \underset{n\rightarrow \infty }{\lim }\sqrt{n}\left( \boldsymbol{%
\widehat{\beta}}_{1,\tau}^{\rm PS} -\bbeta_{1,\tau}\right) \right\}\no \\
&=&\mathbb{E}\left\{ \underset{n\rightarrow \infty }{\lim }\sqrt{n}\left( \boldsymbol{%
\widehat{\beta}}_{1,\tau}^{\rm S} -\left(\boldsymbol{%
\widehat{\beta}}_{1,\tau}^{\rm FM}-\boldsymbol{%
\widehat{\beta}}_{1,\tau}^{\rm SM}\right)\left(1-d\mathcal{W}_{n}^{-1}\right)
\textrm{I}\left(\mathcal{W}_{n}\leq d\right)-\bbeta_{1,\tau}\right) \right\}\no\\
&=&\mathbb{E}\left\{ \underset{n\rightarrow \infty }{\lim }\sqrt{n}\left( \boldsymbol{%
\widehat{\beta}}_{1,\tau}^{\rm S} -\bbeta_{1,\tau}\right) \right\}-\mathbb{E}\left\{ \underset{n\rightarrow \infty }{\lim }\sqrt{n}\left( \boldsymbol{%
\widehat{\beta}}_{1,\tau}^{\rm FM} -\boldsymbol{%
\widehat{\beta}}_{1,\tau}^{\rm SM}\right)\textrm{I}(\mathcal{W}_{n}\leq d) \right\}\no\\
&&+\mathbb{E}\left\{ \underset{n\rightarrow \infty }{\lim }\sqrt{n}\left( \boldsymbol{%
\widehat{\beta}}_{1,\tau}^{\rm FM} -\boldsymbol{%
\widehat{\beta}}_{1,\tau}^{\rm SM}\right)d\mathcal{W}_{n}^{-1}\textrm{I}(\mathcal{W}_{n}<d)\right\}\no\\
&=&d\bdelta\mathbb{E}\left\{\chi _{p_{2}+2}^{-2}(\Delta^2) \right\}+\bdelta \mathbb{H}_{p_{2}+2}\left( d;\Delta \right) -d\bdelta\mathbb{E}\left( \chi _{p_{2}+2}^{-2}(\Delta^2)\textrm{I}\left(\chi _{p_{2}+2}^{2}(\Delta^2)< d\right) \right)\no\\
&=&\bdelta \left\{ d\mathbb{E}\left\{\chi _{p_{2}+2}^{-2}(\Delta^2) \right\}+\mathbb{H}_{p_{2}+2}\left( d;\Delta \right)-d\mathbb{E}\left( \chi _{p_{2}+2}^{-2}(\Delta^2)\textrm{I}\left(\chi _{p_{2}+2}^{2}(\Delta^2)< d\right) \right)\right\}.
\end{eqnarray*}
\end{proof}
\begin{proof}[Proof of Theorem~\ref{risks}]
Firstly, the asymptotic covariance of $\FM$ and $\SM$ are given in Equations (\ref{cov_FM}) and (\ref{cov_SM}).
Now, the asymptotic covariance of $\PT$ is
given by
\begin{eqnarray*}
\boldsymbol{\Gamma} \left( \hbbeta_{1,\tau}^{\rm PT} \right)
&=&\mathbb{E}\left\{ \underset{n\rightarrow \infty }{\lim }{n}\left( \PT -\btau \right)\left( \PT -\btau \right) ^{'}\right\} \\
&=&\mathbb{E}\left\{ \underset{n\rightarrow \infty }{\lim }n\left[ \left( \FM -\tau \right) -\left( \FM  -\SM \right) \textrm{I}\left( %
\mathcal{W}_{n}< c_{n,\alpha }\right) \right] \right. \\
&&\times \left. \left[ \left( \hbbeta_{1,\tau}^{\rm FM} -\bbeta_{1,\tau}\right) -\left( \hbbeta_{1,\tau}^{\rm FM} -\hbbeta_{1,\tau}^{\rm SM} \right) \textrm{I}\left( \mathcal{W}_{n}< \chi^2_{p_2,\alpha}\right) \right] ^{'}\right\} \\
&=&\mathbb{E}\left\{ \left[ \bvt_{1}-\bvt _{3}\textrm{I}\left( \mathcal{W}_{n}<
c_{n,\alpha }\right) \right] \left[ \bvt_{1}-\bvt _{3}\textrm{I}\left( %
\mathcal{W}_{n}< \chi^2_{p_2,\alpha}\right) \right] ^{'}\right\} \\
&=&\mathbb{E}\left\{ \bvt_{1}\bvt_{1}^{'}-2\bvt _{3}\bvt
_{1}^{'}\textrm{I}\left( \mathcal{W}_{n}< \chi^2_{p_2,\alpha}\right) +\bvt
_{3}\bvt_{3}^{'}\textrm{I}\left( \mathcal{W}_{n}< \chi^2_{p_2,\alpha}\right)\right\}
\end{eqnarray*}
Considering, $\mathbb{E}\left\{ \bvt_{1}\bvt_{1}' \right\} = \omega^2\bD_{11.2}^{-1}$
and 
$$
\mathbb{E}\left\{ \bvt_{3}\bvt_{3}^{'}I\left( \mathcal{W}_{n}<
\chi^2_{p_2,\alpha}\right) \right\} = \bPhi \mathbb{H}_{p_{2}+2}\left(
\chi^2_{p_2,\alpha};\Delta \right) +\boldsymbol{\delta \delta }^{'}\mathbb{H}_{p_{2}+4}\left( \chi^2_{p_2,\alpha};\Delta \right)
$$
\begin{eqnarray*}
\lefteqn{\mathbb{E}\left\{ \bvt_{3}\bvt_{1}'I\left( \mathcal{W}_{n}< \chi^2_{p_2,\alpha} \right) \right\}} \\
&=& \mathbb{E}\left\{ \mathbb{E}\left( \bvt_{3}\bvt_{1}^{'}\textrm{I}\left( \mathcal{W}%
_{n} < \chi^2_{p_2,\alpha} \right) |\bvt_{3}\right) \right\}  \\
&=&\mathbb{E}\left\{ \bvt_{3}\mathbb{E}\left( \bvt_{1}^{'}\textrm{I}\left( \mathcal{W}%
_{n}\leq c_{n,\alpha } \right) |\bvt_{3}\right) \right\} \\
&=&\mathbb{E}\left\{ \bvt_{3}\left( \boldsymbol 0+\bSigma_{12}\bPhi^{-1}\left(\bvt_3+\bdelta\right)\right)'\textrm{I}\left( \mathcal{W}%
_{n} < \chi^2_{p_2,\alpha} \right) \right\} \\
&=& \mathbb{E}\left\{ \bvt_{3}\bvt_{3}^{'}\bPhi^{-1}\bSigma_{21}\textrm{I}\left(\mathcal{W}%
_{n}< \chi^2_{p_2,\alpha} \right) \right\}+\mathbb{E}\left\{ \bvt _{3}\bdelta^{'}\bPhi^{-1}\bSigma_{21}\textrm{I}\left(\mathcal{W}%
_{n}< \chi^2_{p_2,\alpha} \right) \right\} \\
&=&\left[ \bPhi \mathbb{H}_{p_{2}+2}\left(
\chi _{p_{2},\alpha }^{2};\Delta \right) +\bdelta \bdelta ^{'}\mathbb{H}_{p_{2}+4}\left( \chi _{p_{2},\alpha
}^{2};\Delta \right)\right]\bPhi^{-1}\bSigma_{21}\\
&& +\bdelta\bdelta^{'}\bPhi^{-1}\bSigma_{21}\mathbb{H}_{p_{2}+2}\left( \chi_{p_{2},\alpha}^{2};\Delta \right)\\
&=& \bSigma_{21} \mathbb{H}_{p_{2}+2}\left(
\chi _{p_{2},\alpha }^{2};\Delta \right)+ \bdelta \bdelta ^{'}\bPhi^{-1}\bSigma_{21}\left[\mathbb{H}_{p_{2}+4}\left( \chi _{p_{2},\alpha
}^{2};\Delta \right)+\mathbb{H}_{p_{2}+2}\left( \chi _{p_{2},\alpha
}^{2};\Delta \right)\right]
\end{eqnarray*}

so finally,
\begin{eqnarray*}\label{cov_PT}
\boldsymbol{\Gamma} \left( \hbbeta_{1,\tau}^{\rm PT} \right)
&=&w^2\bD_{11.2}^{-1}-2\bSigma_{21} \mathbb{H}_{p_{2}+2}\left(
\chi _{p_{2},\alpha }^{2};\Delta \right)\\
&&+ \bdelta \bdelta ^{'}\bPhi^{-1}\bSigma_{21}\left[\mathbb{H}_{p_{2}+4}\left( \chi _{p_{2},\alpha
}^{2};\Delta \right)+\mathbb{H}_{p_{2}+2}\left( \chi _{p_{2},\alpha
}^{2};\Delta \right)\right]\no \\
&&+\bPhi \mathbb{H}_{p_{2}+2}\left(
\chi^2_{p_2,\alpha};\Delta \right) +\boldsymbol{\delta \delta }^{'}\mathbb{H}_{p_{2}+4}\left( \chi^2_{p_2,\alpha};\Delta \right).
\end{eqnarray*}
The asymptotic covariance of $\hbbeta_{1,\tau}^{\rm S}$ is also obtained as
\begin{eqnarray*}
\boldsymbol{\Gamma} \left( \hbbeta_{1,\tau}^{\rm S}\right)
&=&\mathbb{E}\left\{ \underset{n\rightarrow \infty }{\lim }{n}\left( \hbbeta_{1,\tau}^{\rm S}-\bbeta_{1,\tau}\right)\left(
\hbbeta_{1,\tau}^{\rm S}-\bbeta_{1,\tau}\right) ^{'}\right\}  \\
&=&\mathbb{E}\left\{ \underset{n\rightarrow \infty }{\lim }n\left[ \left( \hbbeta_{1,\tau}^{\rm FM}-\bbeta_{1,\tau}\right) -\left(
\hbbeta_{1,\tau}^{\rm FM}-\hbbeta%
_{1,\tau}^{\rm SM}\right) d \mathcal{W}_{n}^{-1}\right] \right.
\\
&&\times \left. \left[ \left( \hbbeta_{1,\tau}^{\rm FM}-\boldsymbol{%
\beta }_{1,\tau}\right) -\left( \hbbeta_{1,\tau}^{\rm FM}-\hbbeta_{1,\tau}^{\rm SM}\right) d \mathcal{W}_{n}^{-1}%
\right] ^{'}\right\}  \\
&=&\mathbb{E}\left\{ \bvt_{1}\bvt_{1}^{'}-2 d
\bvt_{3}\bvt_{1}^{'}\mathcal{W}_{n}^{-1}+d ^{2}\bvt _{3}\bvt_{3}^{'}\mathcal{W}%
_{n}^{-2}\right\} \text{.}
\end{eqnarray*}%
We know that
\begin{equation*}
\mathbb{E}\left( \bvt_{1}\bvt_{1}^{'}\right)=w^2\bD_{11.2}^{-1}
\end{equation*}
Now, we need to compute $\mathbb{E}\left\{ \bvt_{3}\bvt_{3}^{'}\mathcal{W}_{n}^{-2}\right\}$ by using Lemma~\ref{lem_JB}, we have
\begin{equation*}
\mathbb{E}\left\{ \bvt_{3}\bvt_{3}^{'}\mathcal{W}_{n}^{-2}\right\}=\bPhi \mathbb{H}_{p_{2}+2}\left(
\chi _{p_{2},\alpha }^{2};\Delta \right) +\boldsymbol{\delta \delta }^{'}\mathbb{H}_{p_{2}+4}\left( \chi _{p_{2},\alpha
}^{2};\Delta \right)
\end{equation*}
and also,
\begin{eqnarray*}
\mathbb{E}\left\{ \bvt_{3}\bvt_{1}^{'}\mathcal{W}_{n}^{-1}\right\}
&=&\mathbb{E}\left\{ \mathbb{E}\left( \bvt_{3}\bvt_{1}^{'}\mathcal{W}%
_{n}^{-1}|\bvt_{3}\right) \right\} \\
&=&\mathbb{E}\left\{ \bvt_{3}\mathbb{E}\left( \bvt_{1}^{'}\mathcal{W}%
_{n}^{-1}|\bvt_{3}\right) \right\}  \\
&=&\mathbb{E}\left\{ \bvt_{3}\left( \boldsymbol 0+\bSigma_{12}\bPhi^{-1}\left(\bvt_3+\bdelta\right)\right)^{'}\mathcal{W}%
_{n}^{-1} \right\} \\
&=&\mathbb{E}\left\{ \bvt_{3}\left(\bvt_3^{'}-\bdelta^{'}\right)\bPhi^{-1}\bSigma_{21}\mathcal{W}%
_{n}^{-1} \right\} \\
&=&\mathbb{E}\left\{\bvt_3\bvt_3^{'}\bPhi^{-1}\bSigma_{21}\mathcal{W}_{n}^{-1} \right\}+\mathbb{E}\left\{\bvt_3\bdelta^{'}\bPhi^{-1}\bSigma_{21}\mathcal{W}_{n}^{-1}\right\} \\
&=&\bSigma_{21}\mathbb{E}\left\{\chi _{p_{2}+2}^{-2}\left(\Delta\right)\right\}\\
&&+ \bdelta \bdelta'\bPhi^{-1}\bSigma_{21} \left[ \mathbb{E}\left\{\chi _{p_{2}+4}^{-2}\left(\Delta\right)\right\}+\mathbb{E}\left\{\chi _{p_{2}+2}^{-2}\left(\Delta\right)\right\}\right]\\
\end{eqnarray*}%
Therefore, we obtain $\boldsymbol{\Gamma} \left( \SS \right)$ by combining all of the components:
\begin{eqnarray*}\label{cov_S}
\bGamma \left( \SS \right)&=&\omega^{2}\bD_{11.2}^{-1}-2d\left\{\bSigma_{21}\mathbb{E}\left(\chi _{p_{2}+2}^{-2}\left(\Delta\right) \right)+\bdelta\bdelta^{'}\mathbb{E}\left(\chi _{p_{2}+4
}^{-2};\Delta \right)\bPhi^{-1}\bSigma_{21}\right\} \no\\
&&+d^2\left\{\bPhi\mathbb{H}_{p_{2}+2}\left(
\chi _{p_{2},\alpha }^{2};\Delta \right) +\boldsymbol{\delta \delta }^{'}\mathbb{H}_{p_{2}+4}\left( \chi _{p_{2},\alpha
}^{2}\left(\Delta \right)\right)\right\}
\end{eqnarray*}
Finally, we compute  $\boldsymbol{\Gamma} \left( \hbbeta_{1,\tau}^{\rm PS} \right)$:
\begin{eqnarray*}
\boldsymbol{\Gamma} \left( \hbbeta_{1,\tau}^{\rm PS} \right)&=&\mathbb{E}\left\{\underset{n\rightarrow \infty }{\lim }{n}\left( \hbbeta%
_{1,\tau}^{\rm PS} -\bbeta_{1,\tau}\right)\left( \hbbeta%
_{1,\tau}^{\rm PS} -\bbeta_{1,\tau}\right) ^{'}\right\}
\end{eqnarray*}
The covariance of $PS$ :
\begin{eqnarray*}
\hbbeta_{1,\tau}^{\rm PS}&=&\hbbeta_{1,\tau}^{\rm S}-\left(\hbbeta_{1,\tau}^{\rm FM}-\hbbeta_{1,\tau}^{\rm SM}\right)\left(1-d\mathcal{W}_{n}^{-1}\right)\textrm{I}\left(\mathcal{W}_n\leq d\right)
\end{eqnarray*}
\begin{eqnarray*}
\lefteqn{\boldsymbol{\Gamma} \left( \hbbeta_{1,\tau}^{\rm PS} \right)}\\
&=&\mathbb{E}\left\{\underset{n\rightarrow \infty }{\lim }{n}\left( \hbbeta%
_{1,\tau}^{\rm PS} -\bbeta_{1,\tau}\right)\left( \hbbeta%
_{1,\tau}^{\rm PS} -\bbeta_{1,\tau}\right) ^{'}\right\} \\
&=&\mathbb{E}\left\{\underset{n\rightarrow \infty }{\lim }{n}\left[\left(\hbbeta%
_{1,\tau}^{\rm S} -\bbeta_{1,\tau}\right)-\left(\hbbeta_{1,\tau}^{\rm FM}-\hbbeta_{1,\tau}^{\rm SM}\right)\left(1-d\mathcal{W}_n^{-1}\right)\textrm{I}\left(\mathcal{W}_n\leq d\right)\right] \right.
\\
&&\times \left.\left[\left(\hbbeta%
_{1,\tau}^{\rm S} -\bbeta_{1,\tau}\right)-\left(\hbbeta_{1,\tau}^{\rm FM}-\hbbeta_{1,\tau}^{\rm SM}\right)\left(1-d\mathcal{W}_n^{-1}\right)\textrm{I}\left(\mathcal{W}_n \leq d\right)\right]^{'}
\right\} \\
&=&\bGamma \left( \SS \right)\\
&&-2\mathbb{E}\left\{\underset{n\rightarrow \infty }{\lim }{n}\left(\FM-\SM \right)\left(\SS -\btau \right)'\left(1-d\mathcal{W}_n^{-1}\right)\textrm{I}\left(\mathcal{W}_n \leq d\right)\right\} \\
&&+\mathbb{E}\left\{\underset{n\rightarrow \infty }{\lim }{n}\left(\FM-\SM\right)\left(\FM -\SM \right)^{'}\left(1-d\mathcal{W}_n^{-1}\right)^{2}\textrm{I}\left(\mathcal{W}_n \leq d\right)\right\}\\
&=&\bGamma \left( \SS \right)\\
&&-2\mathbb{E}\left\{\underset{n\rightarrow \infty }{\lim }{n}\left(\FM-\SM \right)\left[\left(\FM -\btau\right)-d\left(\FM-\SM \right)\mathcal{W}_{n}^{-1}\right]^{'}\right.\\
&&\left. \times\left(1-d\mathcal{W}_n^{-1}\right)\textrm{I}\left(\mathcal{W}_n\right)\right\} \\
&&+\mathbb{E}\left\{\underset{n\rightarrow \infty }{\lim }{n}\left(\FM-\SM\right)\left(\FM -\SM\right)^{'}\left(1-d\mathcal{W}_n^{-1}\right)^{2}\textrm{I}\left(\mathcal{W}_n\leq d\right)\right\} \\
&=&\bGamma \left( \SS \right)\\
&&-2\mathbb{E}\left\{\bvt_3\bvt_1^{'}\left(1-d\mathcal{W}_n^{-1}\right)\textrm{I}\left(\mathcal{W}_n\leq d\right)-d\bvt_3\bvt_3^{'}\mathcal{W}_n^{-1}\left(1-d\mathcal{W}_n^{-1}\right)\textrm{I}\left(\mathcal{W}_n\leq d\right)\right\} \\
&&+\mathbb{E}\left\{\bvt_3\bvt_3^{'}\left(1-d\mathcal{W}_n^{-1}\right)^{2}\textrm{I}\left(\mathcal{W}_n\leq d\right)\right\} \\
&=&\bGamma \left( \SS \right)-2\mathbb{E}\left\{\bvt_3\bvt_1^{'}\left(1-d\mathcal{W}_n^{-1}\right)\textrm{I}\left(\mathcal{W}_n\leq d\right)\right\}\\
&& +2\mathbb{E}\left\{\bvt_3\bvt_3^{'}d\mathcal{W}_n^{-1}\left(1-d\mathcal{W}_n^{-1}\right)\textrm{I}\left(\mathcal{W}_n\leq d\right)\right\} \\
&&+\mathbb{E}\left\{\bvt_3\bvt_3^{'}\textrm{I}\left(\mathcal{W}_n\leq d\right)\right\} -2\mathbb{E}\left\{\bvt_3\bvt_3^{'}d\mathcal{W}_n^{-1}\textrm{I}\left(\mathcal{W}_n\leq d\right)\right\}\\
&&+\mathbb{E}\left\{\bvt_3\bvt_3^{'}d^2\mathcal{W}_n^{-2}\textrm{I}\left(\mathcal{W}_n\leq d\right)\right\} \\
&=&\bGamma \left( \SS \right)-2\mathbb{E}\left\{\bvt_3\bvt_1^{'}\left(1-d\mathcal{W}_n^{-1}\right)\textrm{I}\left(\mathcal{W}_n\leq d\right)\right\}\\
&& -d^2\mathbb{E}\left\{\bvt_3\bvt_3^{'}\mathcal{W}_n^{-2}\textrm{I}\left(\mathcal{W}_n\leq d\right)\right\}+\mathbb{E}\left\{\bvt_3\bvt_3^{'}\textrm{I}\left(\mathcal{W}_n\leq d\right)\right\}
\end{eqnarray*}
So, we need to the following identities:
 
\begin{eqnarray*}
\mathbb{E}\left\{\bvt_3\bvt_3^{'}\textrm{I}\left(\mathcal{W}_n\leq d\right)\right\}=\bPhi\mathbb{H}_{p_{2}+2}\left(
d;\Delta \right)+\bdelta\bdelta^{'}\mathbb{H}_{p_{2}+4}\left(
d;\Delta \right),
\end{eqnarray*}
\begin{eqnarray*}
\mathbb{E}\left\{\bvt_3\bvt_3^{'}\mathcal{W}_n^{-2}\textrm{I}\left(\mathcal{W}_n\leq d\right)\right\}
&=& \bPhi\mathbb{E}\left[\left(\chi _{p_{2}+2
}^{-4};\Delta \right)\textrm{I}\left(\chi _{p_{2}+2
}^{2}\left(\Delta\right)\leq d \right)\right]\\
&&+\bdelta\bdelta^{'}\mathbb{E}\left[\chi _{p_{2}+2
,\alpha}^{-4}\left(\Delta \right)\textrm{I}\left(\chi _{p_{2}+2
}^{2}\left(\Delta\right)\leq d \right)\right]
\end{eqnarray*} and
\begin{eqnarray*}
\lefteqn{\mathbb{E}\left\{\bvt_3\bvt_1^{'}\left(1-d\bW_n^{-1}\right)\textrm{I}\left(\mathcal{W}_n\leq d\right)\right\}} \\
&=&\mathbb{E}\left\{\mathbb{E}\left[\bvt_3\bvt_1^{'}\left(1-d\bW_n^{-1}\right)\textrm{I}\left(\mathcal{W}_n\leq d\right)|\bvt_3\right]\right\} \\
&=&\mathbb{E}\left\{\bvt_3\mathbb{E}\left[\bvt_1^{'}\left(1-d\mathcal{W}_n^{-1}\right)\textrm{I}\left(\mathcal{W}_n\leq d\right)|\bvt_3\right]\right\} \\
&=&\mathbb{E}\left\{\bvt_3\left[\boldsymbol{0}+\bSigma_{12}\bPhi^{-1}\left(\bvt_3+\bdelta\right)\right]^{'}\left(1-d\bW_n^{-1}\right)\textrm{I}\left(\mathcal{W}_n\leq d\right)\right\} \\
&=&\mathbb{E}\left\{\bvt_3\bvt_3^{'}\bPhi^{-1}\bSigma_{21}\left(1-d\mathcal{W}_n^{-1}\right)\textrm{I}\left(\mathcal{W}_n\leq d\right)\right\} \\
&&+\mathbb{E}\left\{\bvt_3\bdelta^{'}\bPhi^{-1}\bSigma_{21}\left(1-d\bW_n^{-1}\right)\textrm{I}\left(\bW_n\leq d\right)\right\} \\
&=&\bSigma_{21}\mathbb{E}\left\{1-d\chi _{p_{2}+2}^{-2}\left(\Delta\right)\right\}-\bdelta\bdelta^{'}\bPhi^{-1}\bSigma_{21}\mathbb{E}\left\{1-d\chi _{p_{2}+4}^{-2}\left(\Delta\right)\right\}\textrm{I}\left(\chi _{p_{2}+4}^{2}\left(\Delta\right)\leq d\right)\\
&&-\bdelta\bdelta^{'}\bPhi^{-1}\bSigma_{21}\mathbb{E}\left\{1-d\chi _{p_{2}+2}^{-2}\left(\Delta\right)\right\}\textrm{I}\left(\chi _{p_{2}+2}^{2}\left(\Delta\right)\leq d\right)
\end{eqnarray*}
Therefore, we obtain
\begin{eqnarray}\label{cov_PS}
\bGamma \left( \PS \right) &=& \bGamma \left( \SS \right)-2\bSigma_{21}\mathbb{E}\left\{1-d\chi _{p_{2}+2}^{-2}\left(\Delta\right)\right\}\textrm{I}\left(\chi _{p_{2}+2}^{2}\left(\Delta\right)\leq d\right) \no \\
&&+2\bdelta\bdelta^{'}\bPhi^{-1}\bSigma_{21}\mathbb{E}\left\{1-d\chi _{p_{2}+4}^{-2}\left(\Delta\right)\right\}\textrm{I}\left(\chi _{p_{2}+4}^{2}\left(\Delta\right)\leq d\right) \no\\
&&+2\bdelta\bdelta^{'}\bPhi^{-1}\bSigma_{21}\mathbb{E}\left\{1-d\chi _{p_{2}+2}^{-2}\left(\Delta\right)\right\}\textrm{I}\left(\chi _{p_{2}+2}^{2}\left(\Delta\right)\leq d\right) \no \\
&&-d^2\bPhi\mathbb{E}\left\{\chi _{p_{2}+2}^{-4}\left(\Delta\right)\right\}\textrm{I}\left(\chi _{p_{2}+2}^{2}\left(\Delta\right)\leq d\right)\no \\
&&-d^2\bdelta\bdelta^{'}\mathbb{E}\left\{\chi _{p_{2}+2}^{-4}\left(\Delta\right)\right\}\textrm{I}\left(\chi _{p_{2}+2}^{2}\left(\Delta\right)\leq d\right)\no \\
&&+\bPhi\mathbb{H}_{p_{2}+2}\left(
d;\Delta \right)+\bdelta\bdelta^{'}\mathbb{H}_{p_{2}+4}\left(
d;\Delta \right)
\end{eqnarray}
By using the above equations in the definition \eqref{def:risk}, one may directly obtain the asymptotic risks.
\end{proof}

\end{document}